\documentclass[10pt]{amsart}

\usepackage{graphicx, hyperref}
\usepackage{thmtools, thm-restate}
\usepackage{amssymb,latexsym}
\usepackage{a4wide}
\usepackage{psfrag}
\usepackage[all]{xy}

\let\cal\mathcal
\def\AA{{\cal A}}
\def\BB{{\cal B}}
\def\CC{{\cal C}}
\def\DD{{\cal D}}
\def\EE{{\cal E}}
\def\FF{{\cal F}}
\def\GG{{\cal G}}
\def\HH{{\cal H}}
\def\II{{\cal I}}

\def\KK{{\cal K}}

\def\NN{{\cal N}}
\def\OO{{\cal O}}
\def\PP{{\cal P}}

\def\SS{{\cal S}}
\def\TT{{\cal T}}

\let\blb\mathbb

\def\bX{{\blb X}} 
\def\bY{{\blb Y}} 
 
\def\bQ{{\blb Q}}

\def\bP{{\blb P}}

\def\bZ{{\blb Z}}
\def\bS{{\blb S}}

\def\bN{{\blb N}}

\def\bZ{{\blb Z}}

\def\modulo\mod

\def\soc{\operatorname{soc}}

\def\mod{\operatorname{mod}}

\def\coh{\mathop{\text{\upshape{coh}}}}

\def\rad{\operatorname {rad}}

\def\rep{\operatorname{rep}}

\def\Ext{\operatorname {Ext}}

\def\Hom{\operatorname {Hom}}

\def\End{\operatorname {End}}
\def\RHom{\operatorname {RHom}}

\def\im{\operatorname {im}}

\def\coker{\operatorname {coker}}
\def\ker{\operatorname {ker}}

\def\End{\operatorname {End}}

\def\add{\operatorname {add}}
\def\rk{\operatorname {rk}}

\def\Num{\operatorname{Num}}

\def\Kred{\operatorname{\Num}}

\DeclareMathOperator{\ind}{ind}

\newcommand\Db{{D^b}}

\newcommand\tri[3]{#1\to #2\to #3\to #1[1]}

\renewcommand\t{\tau}

\newtheorem{lemma}{Lemma}[section]
\newtheorem{proposition}[lemma]{Proposition}
\newtheorem{theorem}[lemma]{Theorem}
\newtheorem{corollary}[lemma]{Corollary}

\newcounter{MyCounter}

\theoremstyle{definition}

\newtheorem{example}[lemma]{Example}
\newtheorem{definition}[lemma]{Definition}

\newtheorem{construction}[lemma]{Construction}

{

}

\theoremstyle{remark}

\newtheorem{remark}[lemma]{Remark}

\newdimen\uboxsep \uboxsep=1ex
\def\uboxn#1{\vtop to 0pt{\hrule height 0pt depth 0pt\vskip\uboxsep
\hbox to 0pt{\hss #1\hss}\vss}}

\def\uboxs#1{\vbox to 0pt{\vss\hbox to 0pt{\hss #1\hss}
\vskip\uboxsep\hrule height 0pt depth 0pt}}

\def\Ob{\operatorname{Ob}}

\newcommand\exa{\nopagebreak \begin{center}\smallskip \nopagebreak               \begin{minipage}[t]{6cm}\sloppy}
\newcommand\exb{\end{minipage}\kern 1cm\begin{minipage}[t]{8cm}\sloppy}
\newcommand\exc{\end{minipage}\kern -3cm \smallskip\end{center}}

\title{Numerically finite hereditary categories with Serre duality}
\author{Adam-Christiaan van Roosmalen}
\address{Charles University \\ Faculty of Mathematics and Physics \\ Department of Algebra \\ Sokolovsk\'{a} 83 \\ 186 75 Prague 8 \\ Czech Republic}
\email{vanroosmalen@karlin.mff.cuni.cz}
\begin{document}


\bibliographystyle{amsplain}

\begin{abstract}
Let $\AA$ be an abelian hereditary category with Serre duality.  We provide a classification of such categories up to derived equivalence under the additional condition that the Grothendieck group modulo the radical of the Euler form is a free abelian group of finite rank.  Such categories are called numerically finite and this condition is satisfied by the category of coherent sheaves on a smooth projective variety.
\end{abstract}

\maketitle

\tableofcontents

\section{Introduction}

Let $k$ be an algebraically closed field; all categories will be assumed to be $k$-linear.  In this paper, we provide a classification of all numerically finite hereditary categories with Serre duality (see below for definitions).  In this way, we contribute to an ongoing project to classify hereditary categories.

The conditions we impose on our hereditary categories will be of geometrical nature. Recall that the following conditions are satisfied for the category $\coh \bX$ of coherent sheaves on a smooth projective variety $\bX$:
\begin{enumerate}
\item $\AA$ is \emph{Ext-finite}, meaning that for every two elements $A,B \in \AA$, we have $\dim_k \Ext^i(A,B) < \infty$ for each $i \geq 0$.
\item $\AA$ has \emph{finite global dimension}, meaning that $\Ext^n(-,-) = 0$ for some $n \gg 0$.
\item $\AA$ has \emph{Serre duality} \cite{BondalKapranov89}, meaning that there is an additive auto-equivalence $\bS : \Db \AA \to \Db \AA$ such that for every $A,B \in \Ob \Db \AA$ there are isomorphisms
$$\Hom(A,B) \cong \Hom(B,\bS A)^*$$
natural in $A$ and $B$, and where $(-)^*$ is the vector-space dual.
\item $\AA$ is \emph{numerically finite}, meaning that the rank of the free abelian group
$$\Kred \AA = K_0 \AA / \rad \chi(-,-)$$
is finite.  Here, $K_0 \AA$ is the Grothendieck group of $\AA$ and $\chi(-,-)$ is the Euler form.  Numerical finiteness is an important concept in the theory of stability conditions \cite{Brigdeland07}.
\end{enumerate}

Excepting finite global dimension, all of the other conditions are derived invariants.  Thus if $\AA$ satisfies (1),(3), or (4), and $\BB$ is an abelian category such that $\Db \AA \cong \Db \BB$, then $\BB$ satisfies the same condition.  It is known, however, that when $\AA$ has finite global dimension, the same does not need to hold for $\BB$ (\cite{BergVanRoosmalen13}).

The category $\AA$ is called \emph{hereditary} if $\Ext^2(-,-) = 0$ (note that although heredity is not kept under derived equivalence, there is a satisfactory notion of heredity that can be defined for Krull-Schmidt triangulated categories \cite{Ringel05}).  This condition is satisfied when $\AA$ is the category of coherent sheaves on a smooth projective curve.  The theorem below may thus be regarded as providing a classification of noncommutative curves.

\begin{restatable}{maintheorem}{TheoremNumFinite}\label{Theorem:NumFinite}
Let $\AA$ be a nonzero indecomposable hereditary category with Serre duality over an algebraically closed field.  If $\AA$ is numerically finite, then $\AA$ is derived equivalent to either
\begin{enumerate}
\item\label{enumerate:Tube} a tube,
\item\label{enumerate:Representations} the category of finite-dimensional representations of a finite acyclic quiver, or
\item\label{enumerate:WeightedCurve} the category of coherent sheaves of a hereditary $\OO_X$-order where $\bX$ is a smooth projective curve.
\end{enumerate}
\end{restatable}

It is worth noting that these three classes are not mutually disjoint.  Indeed, it is known (see \cite[Appendix A]{ReitenVandenBergh01}) that a category of coherent sheaves on a weighted projective line (in the sense of \cite{GeigleLenzing87}) is equivalent to the category of coherent sheaves of a hereditary $\OO_{\bP^1}$-order, and thus falls into (\ref{enumerate:WeightedCurve}) where the smooth projective curve $\bX$ is a projective line.  However, when this weighted projective line is of domestic type, then the category of coherent sheaves also fits in (\ref{enumerate:Representations}), up to derived equivalence.

The proof of Theorem \ref{Theorem:NumFinite} is involved and consists of three major steps.  One main idea is that when $\AA$ has an exceptional object $E$, one could consider the perpendicular category $E^\perp$ which is, in some sense, a smaller and easier category (this is called \emph{perpendicular induction} \cite{Lenzing07}).  If $\AA$ is numerically finite, then this procedure stops after finitely many steps.  

The first theorem we prove states that if $\AA$ does not contain any exceptional objects, then it contains 1-spherical objects (defined in \cite{Seidel01}).  Since we do not require $\AA$ to be numerically finite, this theorem may be of independent interest.

\begin{restatable}{maintheorem}{theoremLowExt}\label{theorem:LowExt}
Let $\AA$ be a nonzero abelian hereditary Ext-finite category with Serre duality over an algebraically closed field. Then $\AA$ has an object which is either exceptional or 1-spherical.
\end{restatable}

For the proof of Theorem \ref{theorem:LowExt}, we may assume that $\AA$ does not have any exceptional objects.  Since $\AA$ has Serre duality, we know that $\AA$ has Auslander-Reiten sequences and the Auslander-Reiten translation induces an auto-equivalence $\t: \AA \to \AA$.  In the proof of Theorem \ref{theorem:LowExt}, we start with an indecomposable object $X \in \AA$ such that $\dim_k \Ext^1(X,X)$ is minimal.  If $X$ were not 1-spherical, and we can use a nonzero map $X \to \t X$ to construct an object $Y \in \AA$ such that $\dim \Ext^1(Y,Y) < \dim \Ext^1(X,X)$, which contradicts the minimality of $\dim \Ext^1(X,X)$.

Having proven Theorem \ref{theorem:LowExt}, the next step is to consider hereditary categories with Serre duality and without exceptional objects: we know that such a category will have 1-spherical objects.  With each such 1-spherical object, we can associate a twist functor which is an auto-equivalences of the category $\Db \AA$.  Definitions and results about (simple) 1-spherical objects and twist functors are recalled in \S\ref{section:TwistAndTubes}.

In order to prove Theorem \ref{theorem:NoExceptionals} below, we will use these twist functors to construct a \emph{noetherian} hereditary category $\HH$ derived equivalent to the original category $\AA$.  We can then use the classification in \cite{ReVdB02} to obtain the required classification.

\begin{restatable}{maintheorem}{theoremNoExceptionals}\label{theorem:NoExceptionals}
Let $\AA$ be a nonzero indecomposable hereditary category with Serre duality over an algebraically closed field.  If $\AA$ does not have any exceptional objects, then $\AA$ is derived equivalent to either
\begin{itemize}
\item a homogeneous tube, or
\item the category of coherent sheaves on a smooth projective curve of genus at least one.
\end{itemize}
\end{restatable}

We will construct this hereditary noetherian category $\HH$ in \S\ref{subsection:Tilt} using a 1-spherical object $S \in \AA$.  By our construction, $S$ will be a simple object in $\HH$ and, in fact, we will show that $\HH$ has a ``large enough'' set $\SS$ of simple 1-spherical objects such that $\Hom(X,\SS) \not= 0$ for any $X \in \SS$.  Moreover, all length modules of $\HH$ are obtained by taking successive extensions of objects in $\SS$.
 
We then define a torsion pair $(\TT, \FF)$ by taking $\TT \subseteq \HH$ to be all length modules.  As in \cite{ReVdB02}, we will consider the category $\HH / \TT$ in \S\ref{subsection:Quotient} and show that it has a simple object.  This object corresponds to a noetherian object in $\HH$ and has to lie in a noetherian direct summand of $\HH$.

Note that Theorem \ref{theorem:NoExceptionals} requires no conditions on the numerical Grothendieck group.  In order to find other and ``larger'' categories, we need to allow exceptional objects.  If $\AA$ has an exceptional object $E$, we will consider the perpendicular category $E^\perp$.  This category is then another hereditary category with Serre duality which may or may not have other exceptional objects.  This is where the condition on the size of the numerical Grothendieck group enters: if the numerical Grothendieck group has finite rank, then this process stops after finitely many steps.  In other words, the category $\Db \AA$ has a finite maximal exceptional sequence $\EE = E_1, E_2, \ldots, E_n$.

Here, we need to distinguish between two main cases.  The first case is that where the resulting category $\EE^\perp$ is the zero category.  In this case, we know that $\AA$ has a tilting object and we may invoke \cite{Happel01} to obtain the classification.  We can thus restrict our attention to the case where the resulting category is nonzero.

The first step in the proof of Theorem \ref{Theorem:NumFinite} is to consider the case where this process stops after one step, namely the case where $E^\perp$ does not have any exceptional objects anymore.  In this case, we show in Proposition \ref{proposition:LiesInASimpleTube} that $E$ is simple in $\AA$ and that $E \oplus \t E$ is a generalized 1-spherical object.

If the maximal exceptional sequence $\EE = E_1, E_2, \ldots, E_n$ consists of more than one element, we will show in Proposition \ref{proposition:EnoughSimpleTubes} that every exceptional object has a finite $\t$-orbit.  This implies the existence of more generalized 1-spherical objects, and we can again use these to find a derived equivalent hereditary category $\HH$ with a ``large enough'' set $\SS$ of (semi-simple) generalized 1-spherical objects.  The proof of Theorem \ref{Theorem:NumFinite} then proceeds in a similar way as the proof of Theorem \ref{theorem:NoExceptionals}.

We wish to remark that although it is easy to recover the classification of all abelian 1-Calabi-Yau categories (\cite{vanRoosmalen08}) from Theorem \ref{theorem:NoExceptionals} and the classification of hereditary categories with a tilting object (\cite{Happel01}) from Theorem \ref{Theorem:NumFinite}, both of these classification results are used in proving the results in this paper.

There are examples known of hereditary categories with Serre duality which are not numerically finite.  In fact, some of these categories are noetherian (see \cite{ReVdB02}).  For all of the three classes in Theorem \ref{Theorem:NumFinite}, infinite versions have been constructed in \cite{BergVanRoosmalen13, Ringel02, vanRoosmalen12b}.

\textbf{Acknowledgments.}  The author wishes to thank Michel Van den Bergh for many useful discussions and ideas, and Sarah Witherspoon for useful comments on an early version of this paper.  The author also gratefully acknowledges the hospitality and support of the Max-Planck-Instit\"ut f\"ur Mathematik in Bonn.  This material is based upon work supported by the National Science Foundation under Grant No. 0932078 000, while the author was in residence at the Mathematical Science Research Institute (MSRI) in Berkeley, California, during the spring semester of 2013.  This work was partially funded by the Eduard \v{C}ech Institute under grant GA CR P201/12/G028.
\section{Preliminaries and first results}

We fix an algebraically closed field $k$.  All vector spaces, algebras, and categories will be $k$-linear.  Furthermore, we will always work with essentially small categories.  A $k$-linear category $\CC$ is said to be Hom-finite if $\dim_k \Hom_\CC(A,B) < \infty$ for every two objects $A,B \in \CC$.

For a Krull-Schmidt category $\CC$, let $\ind \CC$ be the full subcategory of $\CC$ whose objects are representatives of isomorphism classes of $\CC$.  If $\CC'$ is a full Krull-Schmidt subcategory of $\CC$ closed under isomorphisms, we will choose $\ind \CC'$ to be a subcategory of $\ind \CC$.  Thus for each isomorphism class of indecomposable objects in $\CC'$, we will choose the same representative in $\ind \CC'$ as we had chosen in $\CC$.

\subsection{Hereditary categories}

An abelian category $\AA$ is called hereditary if $\Ext_\AA^2(-,-) = 0$, or equivalently, if and only if $\Ext^1(A,-)$ and $\Ext^1(-,A)$ are right exact for all objects $A \in \AA$ (\cite[Lemma A.1]{ReVdB02}).

We will follow standard notations and conventions about derived categories (see for example \cite{GelfandManin03, Hartshorne66,Keller07}).  Let $\Db \AA$ be the bounded derived category of an abelian category $\AA$.  We will write $X[n]$ for the $n$-fold suspension of $X \in \Db \AA$.

There is a fully faithful function $\AA \to \Db (\AA)$, mapping $A$ to a complex which is the stalk complex of $A$ concentrated in degree 0; we denote this complex by $A[0]$.  Also, we will write $(A[0])[n]$ by $A[n]$.  We will write $\AA[0]$ for the full subcategory of $\Db \AA$ given by the stalk complexes concentrated in degree $0$, and write $\AA[n]$ for $(\AA[0])[n]$.  A quasi-inverse to the functor $\AA \to \AA[0]$ is given by taking the zeroth cohomology $H^0: \AA[0] \to \AA$.  Furthermore, note that $\AA[n] \cong \AA[m]$ as categories, but only as subcategories of $\Db \AA$ if $n=m$.  For $A,B \in \AA$, we have $\Hom_{\Db \AA}(A[m],B[n]) \cong \Ext_{\AA}^{n-m}(A,B)$.

For a hereditary category $\AA$ it is well-known (see for example \cite[5.2 Lemma]{Happel88}) that every object $A \in \Db \AA$ can be written as
$$\prod_{k \in \bZ} H^{-k}A[k] \cong A \cong \coprod_{k \in \bZ} H^{-k}A[k].$$
In particular, $\Hom_{\Db (\AA)}(A,B) \cong \prod_{k,l} \Ext^{l-k}_{\AA}(H^{-k}(A),H^{-l}(B)).$

\subsection{Serre duality and saturation}\label{section:SerreDuality}
Let $\CC$ be a Hom-finite triangulated $k$-linear category.  A \emph{Serre functor} \cite{BondalKapranov89} on $\CC$ is an additive autoequivalence $\bS : \CC \to \CC$ such that for every $X,Y \in \Ob \CC$ there are isomorphisms
$$\Hom(X,Y) \cong \Hom(Y,\bS X)^*,$$
natural in $X$ and $Y$, and where $(-)^*$ is the vector-space dual.

An abelian category $\AA$ is said to satisfy \emph{Serre duality} when $\AA$ is Ext-finite and the bounded derived category $\Db \AA$ admits a Serre functor.

It has been shown in \cite{ReVdB02} that $\CC$ has Serre duality if and only if $\CC$ has Auslander-Reiten triangles.  If we denote the Auslander-Reiten shift by $\t$, then $\bS \cong \t[1]$.

For easy reference, we recall the following result from \cite{ReVdB02}.

\begin{proposition}\label{proposition:SerreDuality}
If $\AA$ is an Ext-finite hereditary category, then $\AA$ satisfies Serre duality if and only if $\AA$ has almost split sequences and there is a one-to-one correspondence between the indecomposable projective objects $P$ and the indecomposable injective objects $I$, such that the simple top of $P$ is isomorphic to the simple socle of $I$.

If $\AA$ is an Ext-finite hereditary category with no nonzero projectives or injective objects, then $\AA$ has Serre duality if and only if $\AA$ has almost split sequences.  In this case, the autoequivalence $\t: \Db \AA \to \Db \AA$ restricts to an autoequivalence $\t: \AA \to \AA$.
\end{proposition}

Let $\CC$ be a Hom-finite triangulated category.  Recall that a (co)homo\-log\-i\-cal functor $H: \CC \to \mod k$ is of \emph{finite type} if $\sum_{i \in \bZ} \dim H(A[n])$ is finite.  We will say that $\CC$ is of \emph{finite type} if the functors $\Hom_\CC(C,-)$ and $\Hom(-,C)$ are of finite type.  When $\AA$ is an Ext-finite abelian category of finite global dimension, then $\Db \AA$ is of finite type.

Following \cite{BondalKapranov89}, we will say that a Hom-finite triangulated category $\CC$ of finite type is \emph{saturated} if every (co)homological functor of finite type is representable.  We will say that an abelian category $\AA$ is saturated if $\Db \AA$ is of finite type and $\Db \AA$ is saturated.  We will use the following proposition from \cite{BondalKapranov89}.

\begin{proposition}\label{proposition:Saturated}
Let $\CC$ and $\DD$ be triangulated categories of finite type.  If $\CC$ is saturated, then every exact functor $\CC \to \DD$ admits a left and a right adjoint.
\end{proposition}

It was shown in \cite[Theorem C]{ReVdB02} (see also \cite{BondalKapranov89}) that the following categories are saturated:
\begin{enumerate}
\item the category $\mod \Lambda$ of finite-dimensional modules over a finite-dimensional algebra $\Lambda$ of finite global dimension, and
\item the category $\coh \OO$ where $\OO$ is a sheaf of hereditary $\OO_\bX$-orders over a smooth projective curve $\bX$.
\end{enumerate}

\subsection{(Numerical) Grothendieck group}

Let $\AA$ be an (essentially small) Ext-finite abelian category of finite global dimension.  The Euler form $\chi(-,-): \Ob \AA \times \Ob \AA \to \bZ$ is defined to be
$$\chi(X,Y) = \sum_{n \in \bN} (-1)^n \dim_k \Ext^n(X,Y).$$
Note that $\chi(-,-)$ is additive in the following sense: if $0 \to X \to Y \to Z \to 0$ is a short exact sequence in $\AA$, then $\chi(A,X) + \chi(A,Z) = \chi(A,Y)$ and $\chi(X,A) + \chi(Z,A) = \chi(Y,A)$, for any $A \in \AA$.

The Grothendieck group $K_0 (\AA)$ of $\AA$ is the abelian group generated by the isomorphism classes of $\Ob \AA$ and with relations $[X] + [Z] = [Y]$ for each short exact sequence $0 \to X \to Y \to Z \to 0$ in $\AA$.  Here we have denoted by $[X]$ the isomorphism class of $[X]$ in $\AA$.

Similar definitions hold for an (essentially small) triangulated category $\CC$ of finite type. The Euler form $\chi(-,-): \Ob \CC \times \Ob \CC \to \bZ$ is given by
$$\chi(X,Y) = \sum_{n \in \bZ} (-1)^n \dim_k \Hom(X,Y[n]),$$
and is additive in the following sense: if $X \to Y \to Z \to X[1]$ is a triangle in $\CC$, then $\chi(A,X) + \chi(A,Z) = \chi(A,Y)$ and $\chi(X,A) + \chi(Z,A) = \chi(Y,A)$, for any $A \in \CC$.

The Grothendieck group $K_0 (\CC)$ of $\CC$ is the abelian group generated by the isomorphism classes of $\Ob \CC$ and with relations $[X] + [Z] = [Y]$ for each triangle $X \to Y \to Z \to X[1]$ in $\CC$, where $[X]$ the isomorphism class of $[X]$ in $\CC$.

We have the following correspondence between $\AA$ and $\Db \AA$ (see \cite[Chapter III.1]{Happel88}).

\begin{proposition}\label{proposition:GrothendieckGroupHappel}
The embedding $\AA \to \Db \AA$ induces an isomorphism $K_0 (\AA) \to K_0 (\Db \AA)$, compatible with the Euler form. 
\end{proposition}

Assume now that $\AA$ has Serre duality, thus there is an exact autoequivalence $\bS: \Db \AA \to \Db \AA$ such that $\chi(X,Y) = \chi(Y,\bS X)$.  This implies that $\chi(v,-) = 0$ if and only if $\chi(-,v)$ = 0, for all $v \in K_0 (\Db \AA)$.  Moreover, there is a $\bZ$-linear transformation $\Phi: K_0 (\Db \AA) \to K_0 (\Db \AA)$, satisfying the property $\chi(v,-) = -\chi(-, \Phi(v))$ for all $v \in K_0 (\Db \AA)$.  Indeed, one can take $\Phi([X]) = [\t X]$ where $\t X \cong \bS X [-1]$.

Using the isomorphism in Proposition \ref{proposition:GrothendieckGroupHappel}, we find that similar properties hold for $K_0 (\AA)$: $\chi(v,-) = 0$ if and only if $\chi(-,v)$ = 0 for all $v \in \AA$, and there is a $\bZ$-linear transformation $\Phi: K_0 (\AA) \to K_0 (\AA)$, satisfying the property $\chi(v,-) = -\chi(-, \Phi(v))$ for all $v \in K_0 (\AA)$.

\begin{remark}\label{remark:Coxeter}
In general, for an object $X \in \AA$, there is no object $Y \in \AA$ such that $\Phi([X]) = [Y]$.  If $\AA$ is a hereditary category and has no nonzero projective and injective objects, then the functor $\t: \Db \AA \to \Db \AA$ restricts to a functor $\t: \AA \to \AA$ and we do have $\Phi([X]) = [\t X]$, for all $X \in \AA$.
\end{remark}

We write 
$$\rad \chi = \{v \in K_0(\AA) \mid \chi(v,-) = 0\} = \{v \in K_0(\AA) \mid \chi(-,v) = 0\}$$
(the last equality is due to Serre duality).  The quotient $K_0(\AA) / \rad \chi$ is called the \emph{numerical Grothendieck group} of $\AA$, and is denoted by $\Num \AA$.  The induced bilinear form $\chi(-,-): \Num \AA \times \Num \AA \to \bZ$ is nondegenerate.  Since $\Phi(\rad \chi) = \rad \chi$, the transformation $\Phi: K_0(\AA) \to K_0(\AA)$ induces a transformation $\Phi: \Num \AA \to \Num \AA$, also satisfying $\chi(v,-) = -\chi(-,\Phi(v))$ for all $v \in \Num \AA$.

It is easy to see that $\Num \AA$ is a torsion-free abelian group.  We say that $\AA$ is \emph{numerically finite} if $\Num \AA$ is finitely generated.  In this case, $\Num \AA$ is a free abelian group of finite rank.

Assume now that $\AA$ is numerically finite and choose a basis ${v_1, \ldots, v_n}$ of $\Num \AA$.  We may now associate matrices to $\chi(-,-)$ and $\Phi$ with respect to this basis, called the \emph{Cartan matrix} and the \emph{Coxeter matrix}, respectively.  We have the following relation (see for example \cite[Proposition 2.1]{Lenzing96}).

\begin{proposition}\label{proposition:CartanCoxeter}
Let $\AA$ be an Ext-finite abelian category of finite global dimension.  Assume furthermore that $\AA$ has Serre duality and is numerically finite.  Let $A$ be the Cartan matrix and $C$ be the Coxeter matrix with respect to a chosen basis of $\Num \AA$.  We have $C = -A^{-1} A^T$.
\end{proposition}

\begin{proof}
Let $V = (\Num \AA) \otimes_\bZ \bQ$.  Writing $v_1, \ldots v_n$ for the chosen basis of $\Num \AA$, we find a basis $v'_1, \ldots, v'_n$ for $V$ by $v'_i = v_i \otimes 1$.  Furthermore, we well consider the bilinear form $\chi'(-,-)$ on $V$ via
$$\chi'(v \otimes q, w \otimes r) = \chi(v,w) \cdot qr$$
and the linear transformation $\Phi': V \to V$ via
$$\Phi'(v \otimes q) = \Phi'(v) \cdot q.$$
Note that $\chi'(-,-)$ is nondegenerate and that $\Phi'$ satisfies $\chi'(v \otimes q, -) = -\chi'(-, \Phi'(v \otimes q))$ for all $v \in \Num \AA$ and $q \in \bQ$.

Let $\Psi$ be the linear transformation of $V$ whose matrix is $C = -A^{-1} A^T$.  It is then easily verified that $\chi'(v \otimes q, -) = -\chi'(-, \Psi(v \otimes q))$ for all $v \in \Num \AA$ and $q \in \bQ$.  Since then $\chi'(-, \Phi'(v \otimes q)) = \chi'(-, \Psi(v \otimes q))$, we can use that $\chi'(-,-)$ is nondegenerate to conclude that $\Phi' = \Psi$.  This concludes the proof.
\end{proof}

\begin{remark}
Since $\chi(-,-)$ is nondegenerate, we know that $\det A \not= 0$, but we do not know whether $\det A$ is invertible in $\bZ$.  Thus, the matrix $A^{-1}$ should be considered as a matrix over $\bQ$.  However, Proposition \ref{proposition:CartanCoxeter} does imply that the entries of $C$ lie in $\bZ$.
\end{remark}

\begin{example}\label{example:CartanCoxeterCurve}
Let $\bX$ be a smooth projective curve and let $\FF, \GG \in \coh \bX$.  The Riemann-Roch formula gives
$$\chi(\FF,\GG) = \begin{pmatrix} \deg \FF & \rk \FF \end{pmatrix}
\begin{pmatrix} 0 & -1 \\1 & 1-g \end{pmatrix}
\begin{pmatrix} \deg \GG \\ \rk \GG \end{pmatrix}.$$
We see that $\Num (\coh \bX) \cong \bZ^2$ and that $[k(P)]$ and $[\OO_\bX]$ form a basis of $\Num (\coh \bX)$ (here, $P \in \bX$ is a closed point, and $k(P)$ is the associated simple sheaf); indeed, for any $\FF \in \coh \bX$, we have $[\FF] = \deg \FF \cdot [k(P)] + \rk \FF \cdot [\OO_\bX]$.  With respect to this basis, the Cartan matrix is
$$\begin{pmatrix} 0 & -1 \\1 & 1-g \end{pmatrix}$$
and the Coxeter matrix is
$$\begin{pmatrix} 1 & 2g-2 \\ 0 & 1 \end{pmatrix}.$$
\end{example}

\subsection{Tilting objects and exceptional sequences}

Let $\AA$ be a Hom-finite abelian category.  We say that an object $E \in \AA$ is \emph{exceptional} if $\Ext^i(E,E) = 0$ whenever $i \not= 0$ and $\Hom(E,E) \cong k$. If $\AA$ is hereditary, then $\Ext^1(E,E) = 0$ implies $\Hom(E,E) \cong k$ when $E$ is indecomposable (\cite[Lemma 4.1]{HappelRingel82}, see also \cite[Proposition 5.1]{Lenzing07}).

Similar notions will be used when working in the bounded derived category $\Db \AA$ instead of $\AA$: an object $E \in \AA$ is \emph{exceptional} if $\Hom(E,E[i]) = 0$ whenever $i \not= 0$ and $\Hom(E,E) \cong k$.

A sequence $\EE = (E_i)_{i=1, \ldots, n}$ of exceptional objects in $\Db \AA$ is said to be \emph{exceptional} if $\Hom_{\Db \AA}(E_i,E_j[l]) = 0$ whenever $i > j$ and all $l \in \bZ$.  An exceptional sequence is called a \emph{strong exceptional sequence} if and only if $\Ext^l(E_i,E_j) = 0$ for all $l \not= 0$. An exceptional sequence $\EE$ is said to be \emph{full} if the smallest triangulated subcategory of $\Db \AA$ which contains $\EE$ is $\Db \AA$ itself.

Let $\AA$ be a hereditary Ext-finite abelian category.  An object $F \in \AA$ is called a \emph{partial tilting object} if $\Ext^1(F,F) = 0$, and $F$ is called a \emph{tilting object} if additionally $\Hom(F,X) = 0 = \Ext^1(F,X)$ implies that $X = 0$.

We will use the following classification (see \cite[Theorem 3.1]{Happel01}).

\begin{theorem}\label{theorem:Happel}
Let $\AA$ be an indecomposable Hom-finite hereditary abelian category.  If $\AA$ has a tilting object, then $\AA$ is derived equivalent to either
\begin{enumerate}
\item the category $\mod \Lambda$ of finite-dimensional modules over a finite-dimensional hereditary algebra $\Lambda$, or
\item the category $\coh \bX$ of coherent sheaves on a weighted projective line $\bX$.
\end{enumerate}
\end{theorem}

The following proposition is well-known to the experts (see \cite[Lemma 5]{CrawleyBoevey93} for the case of the category of modules over a finite-dimensional hereditary algebra).  It will allow us to use the classification (up to derived equivalence) of hereditary categories with a tilting object from Theorem \ref{theorem:Happel} (see for example \cite[Theorem 6.3]{Lenzing07}).

\begin{proposition}\label{proposition:SequenceMeansObject}
Let $\AA$ be a hereditary category, and let $\EE = E_1, E_2, \ldots, E_n$ be an exceptional sequence in $\Db \AA$.  The abelian subcategory $\BB$ generated by $\EE$ is a hereditary category with a tilting object $F$ which is a direct sum of $n$ pairwise nonisomorphic indecomposable objects.  Furthermore, we may order the indecomposable direct summands of $F$ to form an exceptional sequence.
\end{proposition}

\begin{proof}
By possibly taking suspensions, we may assume that $\EE \subseteq \AA$.  Let $\BB$ be the smallest abelian subcategory of $\AA$ containing $\EE$.  Since $\BB$ is closed under extensions, we see that $\BB$ is hereditary.  We will work by induction on the length of the exceptional sequence $\EE$.

Assume that the abelian category $\BB_k$ generated by $E_1, E_2, \ldots, E_k$ has a tilting object given by $F \cong \oplus_{i=1}^k F_i$ (here, the $F_i$'s are pairwise nonisomorphic indecomposable objects).  Let $F_{k+1}$ be the middle term of the universal extension:
$$0 \to E_{k+1} \to F_{k+1} \to \Ext^1(F,E_{k+1}) \otimes_{\End F} F \to 0.$$
Applying the functors $\Hom(F,-)$ and $\Hom(-,F)$ shows that $\Ext^1(F,F_{k+1}) = 0$ and $\Hom(F_{k+1},F)=0$, respectively.  From the long exact sequence obtained by applying $\Hom(E_{k+1},-)$, we find $\Ext^1(E_{k+1},F_{k+1})=0$.  Applying now $\Hom(-,F_{k+1})$ yields $\Ext^1(F_{k+1}, F_{k+1})=0$.

We find that the object $\oplus_{i=1}^{k+1} F_i$ is a tilting object for the category $\BB_{k+1}$ generated by $E_1, E_2, \ldots, E_k, E_{k+1}$.

The last statement then follows from \cite[Theorem A]{AssemSoutoTrepode08}.
\end{proof}

We will use the following proposition about exceptional objects (compare with \cite[Lemma 4.2]{Hubner96}, \cite[Proposition 4.1.1]{Meltzer04}, and \cite[Proposition 5.3]{Lenzing07}).

\begin{proposition}\label{Proposition:ExceptionalInK}
Let $\AA$ be an Ext-finite hereditary category with Serre duality.  An exceptional object in $\AA$ is determined, up to isomorphism, by its image in $\Kred \AA$.
\end{proposition}

\begin{proof}
Let $A$ and $B$ be two exceptional objects in $\AA$ such that $[A]=[B]$ in $\Kred \AA$.  Since $A$ is exceptional, we have $\chi(A,A) = 1$ and since $[A]=[B]$ in $\Kred \AA$, we know that $\chi(A,B) = 1$.  Hence there is a nonzero morphism $f:A \to B$.  Let $I$ be the image of $f$.

Since $I$ is a quotient object of $A$ and $\Ext^1(A,-)$ is right exact (due to heredity), we know there is an epimorphism $\Ext^1(A,A) \to \Ext^1(A,I)$.  Using that $A$ is exceptional, we find that $\Ext^1(A,I) = 0$ and thus $\chi(A,I) > 0$.  It follows from $[A]=[B]$ in $\Kred \AA$ that $\chi(B,I) = \chi(A,I) > 0$ and hence there is a nonzero morphism $B \to I$.  Since $I$ is a subobject of $B$ and $\End(B) \cong k$, we conclude that $B \cong I$.

Likewise, we obtain from $\Ext^1(I,B) = 0$ that $\Hom(I,A) \not= 0$ and thus that $I \cong A$.  We conclude that $A \cong B$.
\end{proof}

\subsection{Paths in Krull-Schmidt categories}

Let $\AA$ and $\BB$ be (essentially small) additive categories.  The coproduct $\AA \oplus \BB$ (as additive categories) has as objects pairs $(A,B)$ with $A \in \AA$ and $B \in \BB$, and a morphism $(A_1, B_1) \to (A_2,B_2)$ is given by a pair $(f,g)$ with $f \in \Hom_\AA(A_1,A_2)$ and $g \in \Hom_\BB(B_1,B_2)$.  The composition is pointwise.  It is clear that $\AA \oplus \BB$ is again an additive category.

An additive category $\AA$ is called \emph{indecomposable} if it is nonzero and it is not the coproduct of two nonzero additive categories.  When $\AA \cong \AA_1 \oplus \AA_2$, where $\AA_1$ is indecomposable, we will call $\AA_1$ a \emph{connected component} of $\AA$.

Although a triangulated category is additive, this notion of indecomposability is too strong for our purposes.  Instead, we will call a triangulated category a \emph{block} if and only if it is nonzero and not the coproduct of two nonzero triangulated categories, thus if it is indecomposable in the category of (essentially small) triangulated categories.  Note that an indecomposable triangulated category is a block, but not necessarily the other way around.  For example, $\Db \mod k$ is a block, but not indecomposable (as additive category).

An additive category is called \emph{Krull-Schmidt} (or Krull-Remark-Schmidt) if every object is a finite direct sum of objects with local endomorphism rings.  This decomposition is then unique up to permuation and isomorphisms of the direct summands.  It is well-known that a Hom-finite additive category is Krull-Schmidt if and only if idempotents split (see for example \cite[Corollary A.2]{ChenYeZhang08}).  If $\AA$ is an abelian Ext-finite category, then $\AA$ and $\Db \AA$ are Krull-Schmidt categories.  Indeed, both $\AA$ and $\Db \AA$ are Hom-finite and idempotents split (idempotents split in $\AA$ because $\AA$ is abelian; that they split in $\Db \AA$ has been shown in \cite[2.10. Corollary]{BalmerSchlichting01}).

Let $\AA$ be an additive Krull-Schmidt category.  An \emph{unoriented path of length $n$} in $\AA$ or $\Db \AA$ is a sequence $X_0, X_1, \ldots, X_n$ of indecomposable objects such that for all $i=0, 1, \ldots, n-1$ we have that either $\Hom(X_i,X_{i+1}) \not= 0$ or $\Hom(X_{i+1},X_i) \not= 0$.  The sequence $X_0, X_1, \ldots, X_n$ of indecomposable objects is called an \emph{(oriented) path of length $n$} if $\Hom(X_i,X_{i+1}) \not= 0$ for all $i=0, 1, \ldots, n-1$.

Let $\CC$ be a triangulated Krull-Schmidt category.  A \emph{suspended path} in $\CC$ is a sequence $X_0, X_1, \ldots, X_n$ of indecomposable objects such that for all $i=0, 1, \ldots, n-1$ we have either $\Hom(X_i,X_{i+1}) \not= 0$ or $X_{i+1} \cong X_i[1]$.

We recall the following result from \cite{vanRoosmalen12b} (based on \cite{HappelZacharia08}).

\begin{theorem}\label{theorem:Paths}
Let $\AA$ be an indecomposable Ext-finite hereditary category, and let $A,B \in \Ob \AA$ be indecomposable objects.  There is an unoriented path of length at most two from $A$ to $B$.  If there is an oriented path from $A$ to $B$, then there is an oriented path from $A$ to $B$ of length at most two.
\end{theorem}

We will also use the following results (see \cite[Lemmas 3 and 6]{Ringel05}).

\begin{lemma}\label{lemma:Ringel}
Let $\CC$ be a triangulated Krull-Schmidt category, and let $X,Z \in \CC$ be indecomposable objects.  If $\Hom_\CC(Z,X[1])$ is nonzero, then there is a path from $X$ to $Z$.
\end{lemma}

\begin{lemma}\label{lemma:Ringel2}
Let $\CC$ be a Krull-Schmidt triangulated category.  If $\CC$ is a block, then either
\begin{enumerate}
\item all nonzero maps between indecomposable objects are invertible, or
\item for every indecomposable object $X \in \CC$, there is a path from $X$ to $X[1]$.
\end{enumerate} 
\end{lemma}

\begin{lemma}\label{lemma:Ringel3}
Let $\AA$ be an Ext-finite abelian category.  The additive category $\AA$ is indecomposable if and only if $\Db \AA$ is a block.
\end{lemma}

\begin{proof}
First, assume that $\AA$ is indecomposable.  Let $A \in \ind \AA$ be any indecomposable object and let $\BB$ be the full subcategory of $\ind \AA$ of all objects $B$ such that there is an unoriented path between $A$ and $B$.  Let $\CC$ be the full subcategory of $\ind \AA$ given by the objects not in $\BB$.  It is then clear that $\AA \cong \add \BB \oplus \add \CC$, where $\add \BB$ and $\add \CC$ are the additive closures of $\BB$ and $\CC$ in $\AA$, repsectively.  Since $\AA$ is indecomposable, we know that $\BB = \ind \AA$.

Let $X \in \Db \AA$, and let $i \in \bZ$ be the largest integer such that $H^i(X)$ is nonzero.  Note that $\Hom(H^i(X) [-i], X) \not= 0$ and \cite[Lemma 5]{Ringel05} implies that $\Db \AA$ is a block.

The other direction is easy, as a decomposition $\AA \cong \AA_1 \oplus \AA_2$ gives a decomposition $\Db \AA \cong \Db \AA_1 \oplus \Db \AA_2$.  Thus, if $\Db \AA$ is a block, then $\AA$ is indecomposable.  
\end{proof}

\subsection{Spanning classes and equivalences between triangulated categories}

We will use techniques and concepts from \cite{Bridgeland99} and \cite{BridgelandKingMiles01}.  Throughout, let $\CC_1$ and $\CC_2$ be Hom-finite triangulated categories with Serre duality; the Serre functors of $\CC_1$ and $\CC_2$ will be denoted by $\bS_1$ and $\bS_2$ respectively.

We will use the following definition from \cite{Bridgeland99}.

\begin{definition}\label{definition:SpanningClass}
A subclass $\Omega$ of the objects of $\CC$ will be called a \emph{spanning class}\index{spanning class}, if for any object $X \in \CC$
\begin{eqnarray*}
 \forall \omega \in \Omega,  \forall i \in \bZ: \Hom(\omega,X[i]) = 0 &\Rightarrow& X \cong 0, \\
 \forall \omega \in \Omega,  \forall i \in \bZ: \Hom(X,\omega[i]) = 0 &\Rightarrow& X \cong 0.
\end{eqnarray*}
\end{definition}

The following result is \cite[Theorem 2.3]{BridgelandKingMiles01}.

\begin{theorem}\label{theorem:SpanningClass}
Let $\CC_1$ and $\CC_2$ be Hom-finite triangulated categories with Serre duality.  Assume $\CC_1$ is nontrivial and $\CC_2$ is indecomposable.  Let $F: \CC_1 \to \CC_2$ be an exact functor, admitting a left adjoint.   If there is a spanning class $\Omega$ of $\CC_1$ such that
$$F: \Hom(\omega_1,\omega_2[i]) \stackrel{\sim}{\rightarrow} \Hom(F\omega_1,F\omega_2[i])$$
is an isomorphism for all $i \in \bZ$ and all $\omega_1,\omega_2 \in \Omega$, and such that $F \bS_1 (\omega) \cong \bS_2 F (\omega)$, then $F$ is an equivalence of categories.
\end{theorem}

We will use the following corollary.

\begin{corollary}\label{corollary:SpanningClass}
Let $\CC_1$ and $\CC_2$ be Hom-finite triangulated categories with Serre duality.  Assume $\CC_1$ is nontrivial and $\CC_2$ is indecomposable.  Let $F: \CC_1 \to \CC_2$ be an exact functor, admitting a left adjoint.  If $F \bS_1 \cong \bS_2 F$, then $F$ is an equivalence of categories.
\end{corollary}

\begin{proof}
Directly from Theorem \ref{theorem:SpanningClass} by taking $\Ob \CC_1$ to be the spanning class of $\CC_1$.
\end{proof}
\section{Existence of exceptional or 1-spherical objects}

Let $\AA$ be a hereditary Ext-finite abelian category with Serre duality.  We denote the Auslander-Reiten translate by $\tau$.  Recall that an object $E$ is called exceptional if $\Ext^1(E,E)=0$ and $\End E \cong k$.  An object $E \in \AA$ is called 1-spherical if $E \cong \t E$ and $\End E \cong k$ (see also Definition \ref{definition:SphericalObject} below). 

The main result of this section is Theorem \ref{theorem:LowExt}, which states that hereditary categories with Serre duality have exceptional objects and/or 1-spherical objects.  The main idea of the proof is to start from an indecomposable object $E$ such that $\dim_k \Ext^1(E,E)$ is minimal.  If $\dim \Ext^1(E,E) = \dim \Hom(E,\t E) \geq 2$ then we will use a nonzero map $E \to \tau E$ to find an object $X$ such that $\dim \Ext^1(X,X) < \dim \Ext^1(E,E)$.

We will say that an object $E$ is \emph{endo-simple} if $\End E \cong k$.  We start by listing a few properties of objects which are not endo-simple.

\subsection{Objects which are not endo-simple}

We start by recalling the following result from \cite[Proposition 3.1]{vanRoosmalen08}.

\begin{proposition}\label{proposition:EndoSimpleExists}
In a $k$-linear Hom-finite abelian category $\AA$, every object $X$ admits an endo-simple object $Y$ which is both a subobject and a quotient object of $X$.
\end{proposition}

\begin{proposition}\label{extraExtension}
Let $\AA$ be an Ext-finite hereditary category.  If $Y$ is an endo-simple object that is both (nontrivially) a subobject and a quotient object of an indecomposable object $X$, then $\dim \Ext^1(Y,Y) \leq \dim \Ext^1(X,X)-1$.
\end{proposition}

\begin{proof}
Consider the short exact sequences
$$\mbox{$0 \to K \to X \to Y \to 0$ and $0 \to Y \to X \to C \to 0$.}$$
We may use these to obtain the following commutative diagram with exact rows and columns:
$$\xymatrix@d{
&&0&0&0& \\
 \Hom(X,Y) \ar[r] & \Hom(Y,Y) \ar[r]^{\varphi} & \Ext^1(C,Y) \ar[r]\ar[u] & \Ext^1(X,Y) \ar[r]\ar[u] & \Ext^1(Y,Y) \ar[r]\ar[u] & 0 \\
 \Hom(X,X) \ar[r] \ar[u] & \Hom(Y,X) \ar[r]\ar[u] & \Ext^1(C,X) \ar[r]^{\mu}\ar[u]^{\psi} & \Ext^1(X,X) \ar[r]\ar[u] & \Ext^1(Y,X) \ar[r]\ar[u] & 0 \\
}$$

Note that there is an epimorphism $\Ext^1(X,X) \to \Ext^1(Y,Y)$.  Also, since $Y$ is endo-simple and $X$ is indecomposable, the map $\Hom(Y,X) \to \Hom(Y,Y)$ is zero.  Let $f \in \Ext^1(C,X)$ be such that $\psi(f) = \varphi(1)$.  Then $\mu(f)$ is a nonzero element in $\Ext^1(X,X)$ which lies in the kernel of $\Ext^1(X,X) \to \Ext^1(Y,Y)$, hence $\dim \Ext^1(Y,Y) \leq \dim \Ext^1(X,X)-1$.
\end{proof}

\subsection{Proof of existence}

Throughout this subsection, let $\AA$ be an Ext-finite hereditary abelian category with Serre duality, and let $E \in \Ob \AA$ be an object of $\AA$ such that $d=\dim \Ext^1(E,E)$ is minimal.  Proposition \ref{extraExtension} yields that $E$ is endo-simple.

Our first step is Proposition \ref{proposition:LowExt} where it is shown that, for the proof of Theorem \ref{theorem:LowExt}, it suffices to prove the existence of an object $X$ with $\dim \Ext(X,X) \leq 1$.  Since part of the proof will be used later, we will give it in a separate lemma.

\begin{lemma}\label{lemma:SmallD}
Let $X \in \Ob \AA$ be an endo-simple object such that $\dim \Ext^1(X,X) = 1$.  If $X$ is not 1-spherical (thus if $\t X \not\cong X$), then $\AA$ has an exceptional object. 
\end{lemma}

\begin{proof}
Consider a nonzero morphism $f : X \to \tau X$; we know that either $\ker f$ or $\coker f$ is nonzero.

Using the epimorphism $X \twoheadrightarrow \im f$, the monomorphism $\im f \hookrightarrow \t X$, and $\dim \Hom(X,\t X) = \dim \Ext^1(X,X) = 1$, we find
\begin{align*}
1 \leq \dim \Hom(X,\im f) &\leq \dim \Hom(X,\t X)  = 1 \\
1 \leq \dim \Hom(\im f, \t X) &\leq \dim \Hom(X,\t X) = 1
\end{align*}
so that $\dim  \Hom(X,\im f) = 1 = \dim \Ext^1(X, \im f)$.

Assume first that $\ker f$ is nonzero.  By applying the functor $\Hom(X,-)$ to the short exact sequence
$$0 \to \ker f \to X \to \im f \to 0,$$
we find:
$$\chi(X,\ker f) = \chi(X,X)-\chi(X,\im f).$$
Since $X$ is endo-simple and $\ker f$ is a subobject of $X$, we know that $\Hom(X,\ker f) = 0$.  This yields 
\begin{eqnarray*}
\dim \Ext^1(X,\ker f) &=& -\dim \Hom(X,X) + \dim \Ext^1(X,X) \\
& & {} + \dim \Hom(X,\im f) - \dim \Ext^1(X,\im f) \\
&=& -1 + 1 + 1 - 1 = 0.
\end{eqnarray*}
Since $\Ext^1(-, \ker f)$ is right exact, we know that $\dim \Ext^1(\ker f , \ker f) \leq \dim \Ext^1(X, \ker f)$, and hence every indecomposable direct summand of $\ker f$ is exceptional.

A dual reasoning shows that, if $\coker f \not\cong 0$, every direct summand of $\coker f$ is exceptional.  This completes the proof.
\end{proof}

\begin{proposition}\label{proposition:LowExt}
If there is an object $X \in \AA$ with $\dim \Ext^1(X,X) \leq 1$, then there is an object $Y \in \Ob \AA$ such that $Y$ is either exceptional or 1-spherical.
\end{proposition}

\begin{proof}
If $\dim \Ext^1(X,X) = 0$, then $X$ is exceptional and we are done.  The case $\dim \Ext^1(X,X) = 1$ follows from Lemma \ref{lemma:SmallD}.
\end{proof}

Recall that $d \geq 0$ is the minimum of $\{\dim \Ext(X,X)\}_{X \in \AA}$.  The following result shows that $d$ is a lower bound for the dimension of certain Hom-spaces.

\begin{proposition}\label{theProposition}
Let $A,B \in \Ob \AA$ and let $d \geq 2$.  If $\Hom(A,B) \not= 0$, then $\dim \Hom(A,\t B) \geq d$.
\end{proposition}

\begin{proof}
Let $f \in \Hom(A,B)$ be nonzero, and consider the epi-mono factorization
$f: A \stackrel{p}{\twoheadrightarrow} \im f \stackrel{i}{\hookrightarrow} B$.  Since for every object $X \in \AA$, we have $\Ext^1(X,X) \geq d$, the category $\AA$ does not have projective objects.  Hence, the autoequivalence $\t: \Db \AA \to \Db \AA$ restricts to an autoequivalence $\t: \AA \to \AA$.  In particular, $\t (i): \t (\im f) \to \t B$ is a monomorphism.  There is then a monomorphism
\begin{eqnarray*}
\Hom(\im f, \t \im f) &\to& \Hom(A,\t B) \\
g &\mapsto& \t (i) \circ g \circ p
\end{eqnarray*}
Because $\dim \Hom(\im f,\t \im f) \geq d$, we find $\dim \Hom(A, \t B) \geq d$.
\end{proof}

\begin{corollary}\label{theCorollary}
Let $d \geq 2$.  If $Y$ is endo-simple, then $\Hom(\t^n Y, Y) = 0$, for all $n > 0$.
\end{corollary}

In the proof of the main theorem, the following lemma will be important.

\begin{lemma}\label{MonoLemma}
Assume $d \geq 2$ and $X$ such that $\dim \Ext^1(X,X) = d$.  Every nonzero map $X \to \t X$ is either a monomorphism or an epimorphism (but not an isomorphism).
\end{lemma}

\begin{proof}
Using the minimality of $d$, Proposition \ref{extraExtension} implies that $X$ is endo-simple.  As before, let $f \in \Hom(X, \t X)$ be nonzero and consider the following two exact sequences
$$0 \to \ker f \to X \to \im f \to 0, \qquad 0 \to \im f \to \t X \to \coker f \to 0.$$
By applying $\Hom(X,-)$ to the first sequence, and $\Hom(-,\t X)$ to the second, we find
\begin{eqnarray*}
\dim \Ext^1(X,\ker f) &=& d-1+\chi(X, \im f) \\
\dim \Ext^1(\coker f, \t X) &=& d-1 - \chi(X,\im f)
\end{eqnarray*}
Since  $\Ext^1(-,\ker f)$ and $\Ext^1(\coker f,-)$ are right exact, we find $\dim \Ext^1(X,\ker f) \geq \dim \Ext^1(\ker f, \ker f)$ and $\dim \Ext^1(\coker f, \t X) \geq \dim \Ext^1(\coker f, \coker f)$.  The above equations then show that either $\dim \Ext^1(\ker f,\ker f) \leq d-1$ or $\Ext^1(\coker f, \coker f) \leq d-1$.  The minimality of $d$ then implies that either $\ker f \cong 0$ or $\coker f \cong 0$, thus $f$ is either a monomorphism or an epimorphism.
\end{proof}

\theoremLowExt*

\begin{proof}
As before, let $X$ be an object such that $d=\dim \Ext^1(X,X)$ is minimal.  By Proposition \ref{extraExtension}, we know that $X$ is endo-simple.  Proposition \ref{proposition:LowExt} yields that it is sufficient to show that $d \leq 1$.  We will assume $d \geq 2$ and obtain a contradiction.  Since in this case, the category $\AA$ cannot have any nonzero projective objects, we know that the autoequivalence $\tau: \Db \to \Db \AA$ restricts to an autoequivalence $\t: \AA \to \AA$.

Let $f \in \Hom(X, \t X)$ be any nonzero morphism.  By Lemma \ref{MonoLemma}, it is either a monomorphism or an epimorphism.  We will assume the former, the latter case is dual.

To ease notations, we will write $Q$ for $\coker f$.  By applying $\Hom(Q,-)$ to the short exact sequence
$$0 \to X \to \t X \to Q \to 0$$
and using $\Hom(Q,\t X) = 0$ because $\t X$ is endo-simple, we find the exact sequence
$$0 \to \Hom(Q,Q) \to \Ext^1(Q,X) \to \Ext^1(Q,\t X) \to \Ext^1(Q,Q) \to 0.$$
From the proof of Lemma \ref{MonoLemma} it follows that
$\dim \Ext^1(Q,\t X) = d-1-\chi(X,X) = 2d-2,$
and hence also $\dim \Ext^1(Q,Q) \leq 2d-2$.  This shows that $Q$ is indecomposable.

By applying $\Hom(X,-)$ to the above sequence, we find $\dim \Hom(X,Q) \not= 0$ such that Proposition \ref{theProposition} yields $\dim \Ext^1(Q,X) = \dim \Hom(X,\t Q) \geq d$.

This gives
\begin{eqnarray*}
\dim \Ext^1(Q,Q) &=& \dim\Ext^1(Q,\t X) - \dim\Ext^1(Q,X) + \dim \Hom(Q,Q)\\
&\leq& (2d-2) - d + \dim \Hom(Q,Q) \\
&=& d-2 + \dim \Hom(Q,Q).
\end{eqnarray*}

If $Q$ is endo-simple, then $\dim\Ext^1(Q,Q) \leq d-1$, and we are done.  If $\dim\Hom(Q,Q) = 2$, then $\dim \Ext^1(Q,Q)=d$ and Propositions \ref{proposition:EndoSimpleExists} and \ref{extraExtension} imply the existence of an object $I$ with $\dim\Ext^1(I,I) \leq d-1$.

Hence, assume that $\dim\Hom(Q,Q)\geq 3$.  Since $Q$ is indecomposable and we are working over an algebraically closed field, we know that $\dim \rad^1(Q,Q) \geq 2$.  Let $I,J$ be the images of two linearly independent noninvertible endomorphisms of $Q$, where $I$ is chosen to be endo-simple.  It follows from Lemma \ref{Different} below that $I$ and $J$ are nonisomorphic either as subobjects of $Q$ or as quotient objects of $Q$.  We will assume the former; if $I \cong J$ as subobjects, then we may apply a dual argument.

If $J$ is endo-simple, we may furthermore assume (possibly by exchanging $I$ and $J$) that the embedding $J \to Q$ does not factor through $I \to Q$.  If $J$ is not endo-simple, then we let $I$ be an endo-simple subobject and quotient object of $J$ (see Proposition \ref{proposition:EndoSimpleExists}).  We may now assume the embedding $J \to Q$ does not factor through $I \to Q$.

Applying $\Hom(J,-)$ to the short exact sequence $0 \to \t I \to \t Q \to \t C \to 0$ yields the exact sequence
\begin{equation}\label{Sequence}
0 \to \Hom(J,\t I) \to \Hom(J,\t Q) \to \Hom(J, \t C) \to \Ext^1(J, \t I).
\end{equation}

We know, due to the choice of $I$, that $\Hom(J,C) \not= 0$, so that Proposition \ref{theProposition} implies that $\dim \Hom(J,\t C) \geq d$.

Since there is an epimorphism $Q \to J$, we may interpret $\Hom(J, \t Q)$ as a subspace of $\Hom(Q, \t Q)$.  Likewise, we may interpret $\Hom(I, \t I)$ as a subspace of $\Hom(Q,\t Q)$.  We can then consider the intersection $V = \Hom(J,\t Q) \cap \Hom(I,\t I)$ as subspaces of $\Hom(Q,\t Q)$, thus a morphism $f \in \Hom(Q, \t Q)$ lies in $V$ if and only if it can be factored as
$$\mbox{$Q \twoheadrightarrow I \to \t I \hookrightarrow \t Q$ and $Q \twoheadrightarrow J \to Q$},$$
for some maps $I \to \t I$ and $J \to Q$.  Note that the image of such a map $f:Q \to \t Q$ in $V$ lies in $\t I$ and thus every map in $V$ can be decomposed as
$$Q \twoheadrightarrow J \to \t I \hookrightarrow \t Q$$
for a certain $J \to \t I$.  Hence, $\dim \Hom(J, \t I) \geq \dim V$. 

Since $\dim \Hom(I,\t I) \geq d$ and $\dim \Hom(Q,\t Q) \leq 2d-2$, we find
\begin{align*}
\dim \Hom(J, \t I) &\geq \dim V \\
& \geq \dim \Hom(J,\t Q) + \dim \Hom(I,\t I) - \dim \Hom(Q,\t Q) \\
& \geq \dim \Hom(J, \t Q)-d+2.
\end{align*}

If $\dim \Hom(I,J) = \dim \Ext^1(J, \t I) \leq 1$, then it follows from the above exact sequence (\ref{Sequence}) that $\dim \Hom(J,\t C) \leq d-1$, contradicting the earlier statement that $\dim \Hom(J,\t C) \geq d$.

Hence, we know that $\dim \Hom(I,J) \geq 2$, and thus also that $\dim \Hom(I,Q) \geq 2$.  In particular, $\Hom(I,C) \not= 0$ so that Proposition \ref{theProposition} implies that $\dim \Hom(I,\t C) \geq d$.

However, applying $\Hom(I,-)$ to the short exact sequence $0 \to \t I \to \t Q \to \t C \to 0$ shows, using that $\Hom(I, \t Q) \leq \Hom(Q, \t Q) \leq 2d-2$:
\begin{eqnarray*}
\dim \Hom(I,\t C) &=& \dim \Hom(I, \t Q) + \dim \Ext^1(I,\t I)  - \dim \Hom(I,\t I) \\
&& {}  + \dim \Ext^1(I,\t C) - \dim \Ext^1(I, \t Q) \\
&\leq& (2d-2) + 1 - d + \dim \Ext^1(I,\t C) - \dim \Ext^1(I, \t Q) \\
&\leq& d-1
\end{eqnarray*}
where we have used that $\dim \Ext^1(I,\t C) \leq \dim \Ext^1(I, \t Q)$ since $\t C$ is a quotient object of $\t Q$.  We obtain a contradiction.
\end{proof}

\begin{lemma}\label{Different}
Let $f,g \in \Hom(Q,Q)$ be nonzero.  Denote $I= \im f$ and $J=\im g$, and assume $I$ is endo-simple.  If $f$ and $g$ are linearly independent, then $I$ and $J$ are not isomorphic as both subobjects and quotient objects of $Q$.
\end{lemma}

\begin{proof}
Seeking a contradiction, assume there are isomorphisms $\alpha,\beta: I \to J$ such that the rightmost and the leftmost triangles commute
$$\xymatrix@R=5pt{
&I \ar@{=}[r]\ar[dd]^{\alpha}&I \ar[rd]^{c}\ar[dd]_{\beta} \\
Q \ar[ur]^{a}\ar[dr]_{b}&&&Q \\
&J \ar@{=}[r]&J \ar[ru]^{d}
}$$
Here, we have $f = c \circ a$ and $g = d \circ b$.  We see that $d \circ b = (c \circ \beta^{-1}) \circ (\alpha \circ a)$.  Using that $I$ is endo-simple (and hence $\beta^{-1} \circ \alpha$ is multiplication by a scalar), we see that the composition of the lower arm is a scalar multiple of the composition of the upper arm.  This contradicts the assumption that $f$ and $g$ are linearly independent.
\end{proof}
\section{Twist functors and tubes}\label{section:TwistAndTubes}

Twist functors and Calabi-Yau objects are essential to the proof presented in this paper. In this section, we recall some relevant definitions and results, and we will prove some additional results that will be used further in the paper.

\subsection{Twist functors and generalized 1-spherical objects}\label{subsection:SphericalObjects}

Twist functors have appeared in the literature under different names, for example \emph{tubular mutations} \cite{Meltzer97}, \emph{shrinking functors} \cite{Ringel84}, and \emph{twist functors} \cite{Seidel01}.  Similar ideas were the mutations used in \cite{GorodentsevRudakov87} in the context of exceptional bundles on projective spaces and, more generally, in \cite{Bondal89}.  We will use a small generalization defined in \cite{vanRoosmalen12} (a similar generalization has been considered in \cite{BurbanBurban05} in a geometric setting).

Let $\AA$ be an Ext-finite abelian category of finite global dimension.  Let $X \in \Db \AA$ be any object and write $A = \End X$.  There is an associated twist functor $T_X: \Db \AA \to \Db \AA$ and a natural transformation $id_{\Db \AA} \to T_X$ such that for every object $Y \in \Db \AA$ there is a triangle
$$\RHom(X,Y) \otimes_A X \to Y \to T_X Y \to (\RHom(X,Y) \otimes_A X)[1].$$

Likewise, there is an associated twist functor $T^*_X: \Db \AA \to \Db \AA$ and a natural transformation $T_X^* \to id_{\Db \AA}$ such that for every object $Y \in \Db \AA$ there is a triangle
$$T^*_X Y \to Y \to \RHom(Y,X)^* \otimes_A X \to T^*_X Y [1].$$

\begin{definition}\label{definition:SphericalObject}
Let $\AA$ be an Ext-finite abelian category with Serre duality.  An object $Y \in \Db \AA$ is called \emph{generalized $n$-spherical} if
\begin{enumerate}
\item $\Hom(Y,Y[i]) = 0$ for all $i \not= 0$ and $i \not= n$,
\item $\Hom(Y,Y)$ is semi-simple, and
\item $\bS Y \cong Y[n]$.
\end{enumerate}
A \emph{minimal $n$-spherical object} is a generalized $n$-spherical object such that no nontrivial direct summands are generalized $n$-spherical objects.

We will say that an object $Y \in \AA$ is minimal or generalized $n$-spherical if the corresponding stalk complex $Y[0] \in \Db \AA$ is minimal or generalized $n$-spherical, respectively.
\end{definition}

\begin{remark}
Note that $\bS X \cong X[n]$ implies that $\Hom(X,X[n]) \cong \Hom(X,X)^*$.
\end{remark}

\begin{remark}\label{remark:TwistWithSelf}
Let $X$ be generalized $n$-spherical.  It is easily checked that $T_X X \cong X[n-1] \cong T^*_X X$.  However, when $n=1$, the natural transformations $id \to T_X$ and $T^*_X \to id$ yield the zero morphisms $X \to T_X X$ and $T^*_X X \to X$.
\end{remark}

\begin{remark}
The difference between a 1-spherical object in the sense of \cite{Seidel01} and a generalized 1-spherical object is that we do not require that $\Hom(X,X) \cong k$.  Generalized 1-spherical object and minimal 1-spherical objects need not be indecomposable.
\end{remark}

\begin{remark}
Minimal 1-spherical objects are minimal Calabi-Yau objects in the sense of \cite{CibilsZhang09}.
\end{remark}

We recall the following result (\cite[Theorem 3.10]{vanRoosmalen12}).

\begin{theorem}\label{theorem:TwistDerived}
Let $\AA$ be an Ext-finite abelian category with Serre duality, and let $E \in \Ob \AA$ be a generalized $n$-spherical object, then the twist functors $T_E$ and $T^*_E$ are quasi-inverses.
\end{theorem}

\begin{remark}\label{remark:TwistNotation}
Due to the previous theorem, we will also write $T^{-n}_Y$ for the functor $(T^*_Y)^n$ when $Y$ is a generalized $n$-spherical object (for $n > 0$).
\end{remark}

\subsection{Perpendicular subcategories}

In this subsection, we will recall some results about perpendicular subcategories and their relation with twist functors.  Let $\AA$ be an abelian Ext-finite hereditary category and let $\SS \subseteq \Ob \AA$.  We will denote by $\SS^\perp$ the full subcategory of $\AA$ given by
$$\Ob \SS^\perp = \{A \in \AA \mid \Ext^i(\SS,X) = 0, \text{ for all } i \geq 0) \}.$$
This subcategory is called the \emph{right perpendicular subcategory to $\SS$}.  It follows from \cite[Proposition 1.1]{GeigleLenzing91} that $\SS^\perp$ is again an abelian hereditary category and that the embedding $\SS^\perp \to \AA$ is exact.  If $\SS = \{E\}$ consists of a single object $E \in \Ob \AA$, then we will also write $E^\perp$ for $\SS^\perp$.

The embedding $\SS^\perp \to \AA$ induces an embedding $\Db_{\SS^\perp} (\AA) \to \Db \AA$.  The essential image of this embedding is given by (see \cite[Lemma 3.6]{ReitenVandenBergh01})
$$\{X \in \Db \AA \mid \Hom(\SS[0],X[n])=0, {\text{ for all } n \in \bZ}\}.$$
We will often identify $\Db (\SS^\perp)$ with this full subcategory of $\Db \AA$.

Assume that $\AA$ has Serre duality.  Let $E \in \Ob \AA$ be an exceptional object (i.e. $\Ext^i(E,E) = 0$ for $i \not= 0$ and $\Hom(E,E) \cong k$).  In this case, the embedding $\Db (E^\perp) \to \AA$ has a left and a right adjoint given by the twist functors $T_E: \Db \AA \to \Db (E^\perp)$ and $T_{\bS E}^{*}: \Db \AA \to \Db (E^\perp)$ respectively.  When $\AA$ is hereditary, these adjoint functors induce adjoint functors $\AA \to E^\perp$ to the embedding $E^\perp \to \AA$ (see \cite{StanleyvanRoomalen12}).

\begin{proposition}\label{proposition:NumAndPerpendicular}
Let $\AA$ be an abelian Ext-finite hereditary category with Serre duality, and let $E \in \AA$ be an exceptional object.  The category $E^\perp$ has Serre duality and $\Num \AA \cong \Num (E^\perp) \oplus \bZ[E]$.
\end{proposition}

\begin{proof}
Since the embedding $\Db (E^\perp) \to \Db \AA$ has a left and a right adjoint, the statement about Serre duality follows from \cite[Lemma 1]{Keller05}.

We know (see \cite[Corollary 4.3]{Lenzing07} or \cite[Corollary 3.8]{ReitenVandenBergh01}) that $K_0 (\AA) \cong K_0 (E^\perp) \oplus \bZ[E]$; the embeddings $K_0 (E^\perp) \to K_0(\AA)$ and $\bZ[E] \to K_0(\AA)$ are given by the natural maps and we will use these maps to consider $K_0 (E^\perp)$ and $\bZ[E]$ as subgroups of $K_0(\AA)$.  We claim that $\rad \chi_{E^\perp} = \rad \chi_{\AA}$.

Let $u \in \rad \chi_{E^\perp}$, and $v \in K_0(E^\perp)$ and $w \in \bZ[E]$.  We have
$$\chi_\AA (v+w, u) = \chi_\AA (v,u) + \chi_\AA (w,u) = \chi_{E^\perp} (v,u) = 0.$$
This shows that $\chi_\AA(-,u) = 0$ and thus $u \in \rad \chi_\AA$.

Next, let $v \in K_0(E^\perp)$ and $w \in \bZ[E]$ such that $v+w \in \rad \chi_\AA$.  Since
$$0 = \chi_\AA([E], v+w) = \chi_\AA([E], w)$$
and $\chi([E], [E]) = 1$, we have $w = 0$.  It then follows from $v \in K_0(E^\perp)$ and  $v \in \rad \chi_\AA$ that $v \in \rad \chi_{E^\perp}$.  We have shown that $\rad \chi_{E^\perp} = \rad \chi_{\AA}$.

The required isomorphism $\Num \AA \cong \Num (E^\perp) \oplus \bZ[E]$ now follows easily.
\end{proof}

\begin{proposition}\label{proposition:PerpendicularOnSpherical}
Let $\AA$ be an abelian Ext-finite hereditary category with Serre duality and without nonzero projective objects.  Let $\SS \subseteq \Ob \AA$ be a set of objects.  If $\SS = \t \SS$, then $\t (\SS^\perp) = \SS^\perp$ and $\t ({}^\perp \SS) = {}^\perp \SS$.   The Serre functor $\bS: \Db \AA \to \Db \AA$ restricts to Serre functors for the triangulated subcategories $\Db_{\SS^\perp}(\AA) = \Db(\SS^\perp) \subseteq \AA$ and $\Db_{{}^\perp \SS}(\AA) = \Db ({}^\perp \SS)$ of $\Db \AA$.
\end{proposition}

\begin{proof}
Since $\AA$ has no nonzero projective objects, Proposition \ref{proposition:SerreDuality} shows that $\t: \AA \to \AA$ is an autoequivalence.

Let $X \in \SS^\perp \subseteq \AA$, thus $\Ext^i(\SS, X) = 0$ for all $i \geq 0$.  We find
$$\Ext^i(\SS, \t^{\pm 1}X) \cong \Ext^i(\t^{\mp 1}\SS, X) = \Ext^i(\SS, X) = 0.$$
This shows that $\t (\SS^\perp) = \SS^\perp$; the equality $\t ({}^\perp \SS) = {}^\perp \SS$ is shown analogously.

Since the Serre functor $\bS$ is given by $\t [1]$, the last claim follows.
\end{proof}

Let $\DD$ be a triangulated category, and let $\BB$ and $\CC$ be full subcategories of $\DD$.  We recall the following proposition (\cite[Lemma 3.1]{Bondal89})

\begin{proposition}\label{proposition:Semiorthogonal}
Let $\DD$ be a triangulated category, and let $\BB$ and $\CC$ be full replete (= closed under isomorphism) triangulated subcategories of $\DD$.  Assume that $\Hom_\DD(\BB, \CC)= 0$.  The following are equivalent:
\begin{enumerate}
\item $\BB$ and $\CC$ generate $\DD$ as triangulated category,
\item every object $D \in \DD$ lies in a triangle $B \to D \to C \to B[1]$ with $B \in \BB$ and $C \in \CC$,
\item $\CC = \BB^\perp$ and the embedding $\BB \to \DD$ has a right adjoint, and
\item $\BB = {}^\perp \CC$ and the embedding $\CC \to \DD$ has a left adjoint.
\end{enumerate}
\end{proposition}

If the pair $(\BB, \CC)$ satisfies the conditions of the previous proposition, we say that $(\BB, \CC)$ is a \emph{semi-orthogonal decomposition} of $\DD$.

\begin{corollary}\label{corollary:SemiOrthogonal}
Let $\AA$ be an abelian Ext-finite category of finite global dimension.  Let $\CC \subseteq \Db \AA$ be a full triangulated subcategory.  If $\CC$ is saturated, then $(\CC, \CC^\perp)$ and $({}^\perp \CC, \CC)$ are semi-orthogonal decompositions of $\Db \AA$.  If, additionally, $\Db \AA$ has a Serre functor $\bS: \Db \AA \to \Db \AA$ and $\bS(\CC) = \CC$, then $\Db \AA = {}^\perp \CC \oplus \CC = \CC \oplus \CC^\perp$.
\end{corollary}

\begin{proof}
Since $\AA$ is an abelian Ext-finite category of finite global dimension, we know that $\Db \AA$, and hence also $\CC$, are of finite type.  By Proposition \ref{proposition:Saturated} we know that the embedding $\CC \to \Db \AA$ has a left and a right adjoint, so that $(\CC, \CC^\perp)$ and $({}^\perp \CC, \CC)$ are semi-orthogonal decompositions of $\Db \AA$.

If furthermore $\Db \AA$ has a Serre functor $\bS: \Db \AA \to \Db \AA$ and $\bS(\CC) = \CC$, then ${}^\perp \CC =\CC^\perp $.  It now follows that $\Db \AA = {}^\perp \CC \oplus \CC = \CC \oplus \CC^\perp$.
\end{proof}

\subsection{Simple tubes}

In this subsection, let $\AA$ be an indecomposable Ext-finite hereditary abelian category with Serre duality.  We will investigate stable components of the Auslander-Reiten quiver of $\AA$ of the form $\bZ A_\infty / \langle \tau^r \rangle$.  Such a stable component $\KK$ is called a \emph{tube}, and we will refer to $r$ as the \emph{rank} of the tube (see Figure \ref{fig:LargerTube2}).  If $r=1$, then $\KK$ is called a \emph{homogeneous tube}.

\begin{figure}
	\centering
		\includegraphics[width=.30\textheight]{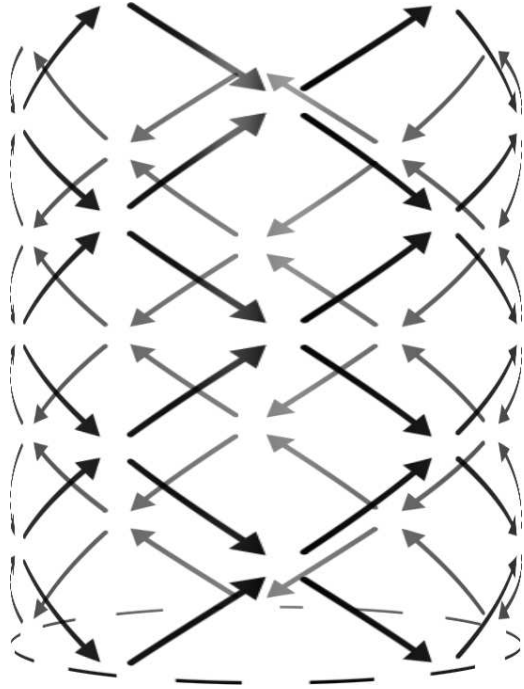}
	\caption{The Auslander-Reiten quiver of a tube of rank 6.}
	\label{fig:LargerTube2}
\end{figure}

For us, a \emph{tube} in $\AA$ will be the additive closure of the indecomposable objects lying in a stable component of the Auslander-Reiten quiver of $\AA$ of the form $\bZ A_\infty / \langle \tau^r \rangle$.

An indecomposable (nonprojective) object $X \in \AA$ is called \emph{peripheral} if the middle term $M$ in the almost split sequence $0 \to \t X \to M \to X \to 0$ is indecomposable (the peripheral objects in the tube in Figure \ref{fig:LargerTube2} are drawn at the bottom).  The number of isomorphism classes of peripheral objects in a tube is given by its rank.  A tube in $\AA$ is called a \emph{simple tube} if all its peripheral objects are simple objects in $\AA$.

We will use the following description from \cite[Proposition 4.4]{vanRoosmalen12} (together with \cite[Theorem 4.2]{vanRoosmalen12}).

\begin{proposition}\label{proposition:AssociatedTubes}\label{proposition:SphericalsInTubes}
A tube in $\AA$ is equivalent to the category of finite-dimensional nilpotent representations of an $\tilde{A}_n$-quiver with cyclic orientation.  The direct sum of a set of representatives of the isomorphism classes of the peripheral objects of a tube is a minimal  1-spherical object, and conversely, every minimal  1-spherical object occurs in this way.
\end{proposition}

\begin{remark}
The previous proposition implies that there is no difference between giving a minimal 1-spherical object or a tube.  We will often change between these two notions.  Given a minimal 1-spherical object $Y$, the associated tube $\TT$ is the abelian subcategory generated by $Y$, thus the full subcategory of $\AA$ of which have a composition series given by direct summands of $Y$.  Conversely, given a tube $\TT$, the associated minimal 1-spherical object is the direct sum of a set of representatives of the isomorphism classes of the peripheral objects.
\end{remark}

This gives another way of looking at a tube: let $\CC$ be the category of finite-dimensional nilpotent representations of an $\tilde{A}_n$-quiver with cyclic orientation; a tube $\TT$ in $\AA$ is the essential image of an embedding $F: \CC \to \AA$ where $F \circ \t_{\CC} \cong \t_\AA \circ F$.

We will use similar definitions when working in $\Db \AA$, thus a tube in $\Db \AA$ is a stable component of the Auslander-Reiten quiver of $\Db \AA$ of the form $\bZ A_\infty / \langle \tau^r \rangle$, or more accurately, the additive closure of the indecomposable objects lying in such a stable component.  The embedding $\AA \to \Db \AA$ maps a tube of $\AA$ to a tube of $\Db \AA$.

The following results are shown in \cite{vanRoosmalen12}.

\begin{theorem}\label{theorem:TubeCriterium}
Let $\AA$ be a hereditary abelian category with Serre duality.  A connected component of the Auslander-Reiten quiver in $\Db \AA$ is a tube if and only if it contains an indecomposable object $X$ such that $\t^r X \cong X$, for $r \geq 1$.
\end{theorem}

\begin{theorem}\label{theorem:TubeMain}\label{theorem:SimpleTube}\label{theorem:DirectingTubes}
Let $\AA$ be a hereditary abelian category with Serre duality.  A tube $\TT$ in $\Db \AA$ is convex in the sense that if there is a path $X_0 \to \cdots \to X_n$ in $\Db \AA$ with $X_0, X_n \in \TT$, then $X_i \in \TT$ for all $i$.
\end{theorem}

As a consequence, we have the following results.

\begin{corollary}\label{corollary:TubesDirecting}
Let $\AA$ be a hereditary category with Serre duality and let $\TT$ be a tube in $\Db \AA$.  Let $A,B \in \Db \AA$ be indecomposable objects.  If $A$ lies in $\TT$ and $B[n]$ does not lie in $\TT$ for all $n \in \bZ$, then:
\begin{enumerate}
\item $\Hom(A,B) \not= 0$ implies that $\Hom(A,B[n]) = 0$ for all $n \not= 0$, and
\item $\Hom(B,A) \not= 0$ implies that $\Hom(B,A[n]) = 0$ for all $n \not= 0$.
\end{enumerate}
\end{corollary}

\begin{proof}
We will only show the first claim.  Seeking a contradiction, assume that $\Hom(A,B) \not= 0$ and that $\Hom(A,B[n]) \not= 0$ for some $n \not= 0$.  We first prove the claim for $n > 0$.

It follows from Lemma \ref{lemma:Ringel} that $\Hom(A,B[n]) \not= 0$ implies that there is a path from $B[n-1]$ to $A$, and it follows from Lemma \ref{lemma:Ringel2} that there is a path from $B$ to $B[n-1]$.  Since $\Hom(A,B) \not= 0$, we find a path from $A$ to $B$ and back to $A$, so that Theorem \ref{theorem:DirectingTubes} yields that $A$ and $B$ lie in the same tube.  However, we have assumed that $B$ does not lie in $\TT$; this gives the required contradiction.

In the case where $n < 0$, we consider the object $B' = B[n]$.  If $\Hom(A,B') \not= 0$, then the first part of the proof yields that $\Hom(A,B'[-n]) = 0$.  Hence $\Hom(A,B) = \Hom(A,B'[-n]) \not= 0$ implies $\Hom(A,B[n]) = \Hom(A,B') = 0$.
\end{proof}

\begin{lemma}\label{lemma:WhenInTube}
Let $\AA$ be a hereditary category with Serre duality, and let $Y \in \Db \AA$ be a minimal 1-spherical object.  An indecomposable $A \in \Db \AA$ lies in the tube containing $Y$ if and only if $\Hom(A,Y) \not= 0$ and $\Hom(Y,A) \not= 0$. 
\end{lemma}

\begin{proof}
If $\Hom(A,Y) \not= 0$ and $\Hom(Y,A) \not= 0$, then it follows Theorem \ref{theorem:DirectingTubes} that $A$ lies in the tube containing $Y$.  The other implication follows from \cite[Proposition 4.4]{vanRoosmalen12}. 
\end{proof}

We will now turn our attention to simple tubes.  Our goal is to prove Proposition \ref{proposition:Summary} which gives some useful criteria for a tube to be simple.  It states that simple tubes lie either ``at the beginning'' or ``at the end'' of the abelian category (meaning that there are either no nonzero maps going into the tube, or no nonzero maps going out of the tube).  Furthermore, to check whether a tube is simple or not, it suffices to see whether the twist functors associated to that tube restrict from autoequivalences of the derived category to autoequivalences of the abelian category.  Even more strongly, after excluding the trivial cases, one only needs to check that the orbit of a single object under the twist functors lies in the abelian category.

Recall that $\AA[0]$ is the full subcategory of $\AA$ consisting of all stalk complexes concentrated in degree zero.  There is a fully faithful functor $\AA \cong \Db \AA$ whose image is given by $\AA[0]$, mapping an object $A \in \AA$ to a stalk complex $A[0] \in \AA[0] \subset \Db \AA$.  A quasi-inverse is given by taking the zeroth cohomology $H^0: \AA[0] \to \AA$.

We start with the following lemmas.  Recall from Remark \ref{remark:TwistNotation} that we write $T_Y^n$ for $(T^*_Y)^{-n}$ when $n < 0$.

\begin{lemma}\label{key}
Let $\AA$ be an Ext-finite abelian hereditary category with Serre duality.  Let $Y$ be a minimal 1-spherical object in $\Db \AA$, and $A, B \in \Db \AA$ be indecomposable objects such that
\begin{enumerate}
\item $\Hom(A,Y[z]) \not= 0$ if and only if $z = 0$, and
\item $\Hom(B,Y[z]) \not= 0$ if and only if $z = 0$,
\end{enumerate}
then $\Hom(T_Y^{i} A,B) \not= 0$ for $i \ll 0$.
\end{lemma}

\begin{proof}
Since $Y$ is generalized $1$-spherical, we have $\Hom(Y,-) \cong \Hom(-,Y[1])^*$.  Let $n < 0$.  We start with the triangle
$$T^{n-1}_Y A \to T^{n}_Y A \to \RHom(T_Y^{n} A, Y )^* \otimes Y \to T^{n-1}_Y A[1].$$
By Corollary \ref{corollary:TubesDirecting}, we see that $\RHom(T_Y^{n-1} A, Y )$ only has nonzero homology in degree zero, and by Remark \ref{remark:TwistWithSelf} we know that $\Hom(T_Y^{n-1} A, Y )^* \cong \Hom(A, T_Y^{-n+1} Y )^* \cong \Hom(A,Y)^*$.

By applying $\Hom(-,B)$ to $T^{n-1}_Y A \to T^{n}_Y A \to \Hom(A,Y)^* \otimes Y \to T^{n-1}_Y A[1]$ and using that $\Hom(Y[-2],B)\cong \Hom(B,Y[-1])^* = 0$, we see that the map $\Hom(T^{n}_Y A,B[1]) \to \Hom(T^{n-1}_Y A,B[1])$ is an epimorphism, for all $n \in \bZ$.  This gives the following descending sequence
$$\dim \Hom(A, B[1]) \geq \dim \Hom(T^{-1}_Y A, B[1]) \geq \dim \Hom(T^{-2}_Y A, B[1]) \geq \cdots$$

Hence, there is an $N < 0$ such that $\dim \Hom(T^{i}_Y A,B[1]) = \dim \Hom(T^{i-1}_Y A,B[1])$ for all $i \leq N$.  For every such $i$, we have an exact sequence
$$0 \to \Hom(T^{i}_Y A,B) \to \Hom(T^{i-1}_Y A,B) \to \Hom(\Hom(A,Y)^* \otimes Y,B[1]) \to 0.$$
Since $\Hom(Y,B[1]) \cong \Hom(B,Y)^* \not= 0$ and $\Hom(A,Y) \not= 0$, the required property follows immediately from this short exact sequence.
\end{proof}

\begin{lemma}\label{lemma:ClosedUnderTwist}
Let $\AA$ be a hereditary abelian category with Serre duality and let $Y \in \AA[0] \subset \Db \AA$ be a minimal 1-spherical object.  Let $A \in \AA[0]$ be indecomposable.
\begin{enumerate}
\item If $\Hom(A,Y) \not= 0$, then $T_Y A \in \AA[0]$.  Furthermore, $T^*_Y A \in \AA[0]$ or $T^*_Y A \in \AA[-1]$.
\item If $\Hom(Y,A) \not= 0$, then $T^*_Y A \in \AA[0]$.  Furthermore, $T_Y A \in \AA[0]$ or $T_Y A \in \AA[1]$.
\end{enumerate}
\end{lemma}

\begin{proof}
Since $T_Y$ and $T_Y^*$ are autoequivalences of $\Db \AA$ (Theorem \ref{theorem:TwistDerived}), we know that $T_Y A$ and $T_Y^* A$ are indecomposable as well.  We only consider the first statement, the other statement can be proven in a similar fashion.  Thus, assume that $\Hom(A,Y) \not= 0$.

If additionally $\Hom(Y,A) \not= 0$, then Lemma \ref{lemma:WhenInTube} shows that $A$ lies in the tube containing $Y$.  Since this tube corresponds to an abelian subcategory of $\AA = H^0(\Db \AA)$, both $H^0 (T_Y A)$ and $H^0 (T^*_Y A)$ are again nonzero objects in the same tube and hence $T_Y A, T^*_Y A \in \AA[0]$.

We may thus assume that $\Hom(Y,A) = 0$.  Taking the cohomologies of the triangle
$$\RHom(Y,A) \otimes Y \to A \to T_Y A \to (\RHom(Y,A) \otimes Y)[1]$$
yields the exact sequence
$$H^0 A \to H^0 (T_Y A) \to \Hom(Y,A[1]) \otimes H^0 Y \to 0.$$
Using that $Y$ is generalized 1-spherical, we find $\Hom(Y,A[1]) \cong \Hom(A,Y)^* \not= 0$ and hence $H^0 (T_Y A) \not= 0$.  Furthermore, since $\AA$ is hereditary and $T_Y A$ is indecomposable, the cohomology can only be nonzero in a single degree, and thus $T_Y A \in \AA[0]$.

Consider now the triangle
$$T^*_Y A \to A \to \RHom(A,Y)^* \otimes Y \to T^*_Y A [1].$$
Using that $H^i A = 0$ and $H^i Y = 0$ for all $i \not= 0$, the long exact sequence obtained from taking cohomologies of the above triangle shows that $H^i(T^*_Y A) = 0$ for all $i \not= 0$ or $i \not= -1$.  Since $T^*_Y A$ is indecomposable and $\AA$ is hereditary, the cohomology of $T^*_Y A$ can only be nonzero in a single degree.  If $H^0(T^*_Y A) = 0$, then $T^*_Y A \in \AA[0]$; if $H^{-1}(T^*_Y A) = 0$, then $T^*_Y A \in \AA[-1]$.
\end{proof}

\begin{proposition}\label{proposition:Summary}
 The following are equivalent for a minimal 1-spherical object $Y \in \AA[0] \subset \Db \AA$.
\begin{enumerate}
  \item The object $H^0 Y \in \AA$ is semi-simple.
  \item If there are indecomposables $A,B \in \AA$ with $\Hom(A[0],Y) \not= 0$ and $\Hom(Y,B[0]) \not= 0$, then $A,B$ lie in the tube containing $H^0 Y$.
  \item For every $A \in \AA[0]$ with $\Hom(A,Y) \not= 0$ and $\Hom(Y,A) = 0$, we have that $T^{i}_Y A \in \AA[0]$, for all $i \in \bZ$.
    \item For every $A \in \AA[0]$ with $\Hom(Y,A) \not= 0$ and $\Hom(A,Y) = 0$, we have that $T^{i}_Y A \in \AA[0]$, for all $i \in \bZ$.
\end{enumerate}
Furthermore, these properties hold if one of the following two properties hold:
\begin{enumerate}
\item[$(a)$] there is an object $A \in \AA[0]$  with $\Hom(A,Y) \not= 0$ and $\Hom(Y,A) = 0$, and additionally $T^{i}_Y A \in \AA[0]$, for all $i \in \bZ$,
\item[$(b)$] there is an object $A \in \AA[0]$  with $\Hom(Y,A) \not= 0$ and $\Hom(A,Y) = 0$, and additionally $T^{i}_Y A \in \AA[0]$, for all $i \in \bZ$.
\end{enumerate}
\end{proposition}

\begin{proof}
The implication $(1) \Rightarrow (2)$ has been shown in \cite[Proposition 4.7]{vanRoosmalen12}.  For the implication $(2) \Rightarrow (1)$, assume that $H^0 Y$ is not semi-simple, thus there is an indecomposable direct summand $Y'$ of $H^0 Y$ which is not simple.  Let $A$ be a (nontrivial) indecomposable subobject of $Y'$.  Since $Y'$ is a peripheral object in a tube, $Y'$ is simple in the additive category generated by this tube (recall from Proposition \ref{proposition:SphericalsInTubes} that this additive subcategory is an abelian subcategory of $\AA$ which is equivalent to the category of finite-dimensional nilpotent representations of a cyclic quiver).  In particular, we know that $A$ cannot lie in that tube.  Let $B$ be an indecomposable direct summand of $Y' / B$.  Similarly, $B$ does not lie in the tube.  These choices of $A$ and $B$ contradict the statement in $(2)$.

We will prove the implication $(2) \Rightarrow (3)$.  Let $A \in \AA[0]$ with $\Hom(A,Y) \not= 0$ and $\Hom(Y,A) = 0$.  We wish to show that $T^{i}_Y A \in \AA[0]$ for all $i \in \bZ$.  We may assume that $A$ is indecomposable.  It follows from Lemma \ref{lemma:ClosedUnderTwist} that $T_Y^i A \in \AA[0]$ for $i \geq 0$.  

For $i < 0$, we consider the triangle
$$T_Y^{-1} A \to A \to \RHom(A,Y)^* \otimes Y \to T_Y^{-1} A[1].$$
It follows from Lemma \ref{lemma:ClosedUnderTwist} that either $T^{-1}_Y A \in \AA[0]$ or $T^{-1}_Y A[1] \in \AA[0]$.  We want to show the former.  Seeking a contradiction, assume the latter.

Corollary \ref{corollary:TubesDirecting} shows that $\RHom(A,Y)^* \otimes Y \cong \Hom(A,Y)^* \otimes Y$, and hence $\Hom(Y,A) = 0$ implies that the map $A \to \RHom(A,Y)^* \otimes Y$ is not a split epimorphism.  Thus, the map $\Hom(A,Y)^* \otimes Y \to T_Y^{-1} A[1]$ is nonzero.  Since $H^0 A$ does not lie in the tube containing $H^0 Y$ (see Lemma \ref{lemma:WhenInTube}), neither does $H^0 (T^{-1}_Y A[1])$.  This contradicts $(2)$.  Hence, $T^{-1}_Y A \in \AA[0]$, and thus also $T^{i}_Y A \in \AA[0]$ for all $i < 0$, as required.

We will now show $(3) \Rightarrow (2)$.  Let $A, B \in \AA$ such that $\Hom(A[0],Y) \not= 0$ and $\Hom(Y,B[0]) \not= 0$.  We wish to show that either $A$ or $B$ lie in the tube containing $H^0 Y$.  Without loss of generality, we may assume that $A$ does not lie in the tube containing $Y$.

Note that $\Hom(A[z], Y) = 0$ for all $z \not= 0$.  Indeed, when $z \not= 1$ this follows from the heredity of $\AA$ while the case $z = 1$ is covered by Corollary \ref{corollary:TubesDirecting}.

Since $\bS Y \cong Y[1]$, we know that $\Hom(B[-1], Y) \not= 0$.  Note that $B[-1] \in \AA[-1]$ so that the heredity of $\AA$ implies that $\Hom(B[z], Y) = 0$ for $z \not\in \{0,-1\}$.

If we assume that $\Hom(B[0], Y) = 0$, then Lemma \ref{key} would imply $\Hom(T^i_Y (A[0]), B[-1]) \not= 0$ for $i \ll 0$.  Since we have assumed that $T^i_Y (A[0]) \in \AA[0]$ for all $i \in \bZ$, this would give a nonzero morphism from $\AA[0]$ to $\AA[-1]$, a contradiction.  Hence, $\Hom(B[0], Y) \not= 0$.  Lemma \ref{lemma:WhenInTube} shows that $B$ lies in the tube containing $H^0 Y$.

Showing $(2) \Leftrightarrow (4)$ is dual to $(2) \Leftrightarrow (3)$.

Next, we show that $(a) \Rightarrow (3)$.  Let $A \in \AA[0]$ be as in the statement of $(a)$.  Let $A' \in \AA[0]$ be such that $\Hom(A',Y) \not= 0$ and $\Hom(Y,A') = 0$.  We want to show that $T^{i}_Y A' \in \AA[0]$ for all $i \in \bZ$.

Lemma \ref{lemma:ClosedUnderTwist} shows that $T^{i}_Y A' \in \AA[0]$ for all $i \geq 0$.  We thus only need to consider the case $i < 0$.  Seeking a contradiction, let $i < 0$ be the largest $i$ such that $T^{i}_Y A' \not\in \AA[0]$, thus (by Lemma \ref{lemma:ClosedUnderTwist}) $T^{i}_Y A' \not\in \AA[-1]$.

As before, it follows from Corollary \ref{corollary:TubesDirecting} and Lemma \ref{lemma:WhenInTube} that $\Hom(A,Y[z]) = \Hom(A',Y[z]) = 0$ if and only if $z \not= 0$.  Lemma \ref{key} implies that $\Hom(T_Y^{j} A, T_Y^{i} A') \not= 0$ for $j \ll 0$.  Since we have assumed that $T_Y^{j} A \in \AA[0]$, this would give a map from $\AA[0]$ to $\AA[-1]$ and we have obtained a contradiction.

The proof of $(b) \Rightarrow (4)$ is similar.
\end{proof}

\begin{corollary}
Let $\AA$ be a hereditary category with Serre duality.  A minimal 1-spherical object $Y \in \AA$ is semi-simple if and only if the twist functors $T_Y, T_Y^*: \Db \AA \to \Db \AA$ restrict to autoequivalences $\AA[0] \to \AA[0]$.
\end{corollary}

The following propositions will help us to find simple tubes.

\begin{proposition}\label{proposition:PerpendicularOfSimples}
Let $\SS$ be a set of simple tubes in $\AA$.  The embedding $\SS^\perp \to \AA$ maps a simple tube in $\SS^\perp$ to a simple tube in $\AA$.
\end{proposition}

\begin{proof}
Note that $\SS^\perp = {}^\perp \SS$.  Proposition \ref{proposition:PerpendicularOnSpherical} implies that the embedding $\SS^\perp \to \AA$ maps a tube $\TT$ in $\SS^\perp$ to a tube $\TT$ in $\AA$.  We thus need only that check that $\TT$ is simple in $\AA$ if it is simple in $\SS^\perp$.

Let $Y$ be a minimal 1-spherical object of the tube $\TT$ (if $\TT$ has rank $r$ and $E$ is a peripheral object in $\TT$, then $Y \cong \oplus_{i=1}^r \tau^i E$).  Seeking a contradiction, assume that $\TT$ is not simple in $\AA$.  According to Proposition \ref{proposition:Summary}, we may assume that there is an indecomposable object $A \in \AA$ such that $\Hom(A,Y) \not= 0$ and $\Hom(Y,A) = 0$, and such that $T^{i}_Y (A[0]) \not\in \AA[0]$, for some $i \in \bZ$.

Since $\TT$ is a simple tube in $\SS^\perp$, this object $A$ cannot lie in $\SS^\perp$.  There is thus a tube $\TT_S \in \SS$ such that either $\Hom(\TT_S,A) \not= 0$ or $\Hom(A, \TT_S) \not= 0$.  We will start by considering the former.

Since $Y \in \SS^\perp$, we know that $\Hom(\TT_S, T^{j}_Y A) \not= 0$ for all $j \in \bZ$.  However, Proposition \ref{proposition:Summary} shows that for some $i \in \bZ$, we have $T^{i}_Y A \not\in \AA[0]$ but rather $T^{i}_Y A \in \AA[-1]$ (Lemma \ref{lemma:ClosedUnderTwist}).  This gives a nonzero map from $\TT_S \subseteq \AA[0]$ to $T^{i}_Y A \in \AA[-1]$, which is impossible.

Assume now that $\Hom(A, \TT_S) \not= 0$.  Similarly, we find an $i \in \bZ$ such that $\Hom(T^{i}_Y A, \TT_S) \not= 0$ and where $T^{i}_Y A \in \AA[-1]$.  Let $X$ be a minimal 1-spherical object of the tube $\TT_S$ (if $\TT_S$ has rank $s$ and $S$ is a peripheral object in $\TT_S$, then $X \cong \oplus_{i=1}^s \tau^i S$). By Serre duality, we have
$$0 \not= \Hom(T^{i}_Y A, X) \cong \Hom(\tau^{-1} X, T^{i}_Y A [1])^* \cong \Hom(X, T^{i}_Y A [1])^*.$$
This contradicts our assumption that $\TT_S \in \SS$ is a simple tube (see Proposition \ref{proposition:Summary}).  We conclude that $\TT$ is a simple tube in $\AA$.
\end{proof}

The previous proposition does not hold if $\SS$ does not consist of simple tubes.

\begin{example}
Let $Q$ be the Kronecker quiver and let $\rep_k Q$ be the category of finite-dimensional $k$-representations of $Q$.  A description of this category can be found in \cite[Section VIII.7]{ARS}.  It follows from \cite[Propositions 7.1 and 7.4]{ARS} that all indecomposable regular objects lie in tubes, but none of these tubes are simple.  Let $\SS$ be such a tube.  It follows from \cite[Theorem 7.5]{ARS} that the category $\SS^\perp$ is a direct sum of tubes, and hence all these tubes are simple in $\SS^\perp$.  However, the embedding $\SS^\perp \to \rep_k Q$ maps none of these tubes to simple tubes in $\rep_k Q$.
\end{example}

\begin{proposition}\label{proposition:OrthogonalAndPathMeansSimple}
Let $\TT_1$ and $\TT_2$ be two tubes in $\AA$.  Assume that $\TT_2 \in \TT_1^\perp$.  If there is a path from $\TT_1$ to $\TT_2$, then $\TT_1$ and $\TT_2$ are simple tubes.
\end{proposition}

\begin{proof}
Let $Y_1$ and $Y_2$ be the minimal 1-spherical objects corresponding to the tubes $\TT_1$ and $\TT_2$, respectively.

Let $T_1 \to A_1 \to A_2 \to \cdots \to A_n \to T_2$ be a path of minimal length between an indecomposable object $T_1 \in \TT_1$ and $T_2 \in \TT_2$.  This implies that $\Hom(\TT_1, A_2) = 0$ and thus that $T^{i}_{Y_1} A_2 \cong A_2$ for all $i \in \bZ$.  Seeking a contradiction, assume that $\TT_1$ is not simple.  According to Proposition \ref{proposition:Summary}, we know that $T^{i}_{Y_1} A_1 \not\in \AA[0]$ for some $i \in \bZ$ and by Lemma \ref{lemma:ClosedUnderTwist} we may assume that $i > 0$. Let $j \geq 0$ be the largest integer such that $T^{j}_{Y_1} A_1 \in \ind \AA[0]$.  Then there is an epimorphism
$$\Hom(Y_1,A_1) \otimes Y_1 \to T^{j}_{Y_1} A_1$$
and hence, using that $T^{j}_{Y_1} A_2 \cong A_2$, also a nonzero morphism $Y_1 \to A_2$, contradicting the minimality of the path.
\end{proof}

If we assume that $\AA$ is indecomposable, we have the following, stronger, proposition.

\begin{proposition}\label{proposition:OrthogonalToSimpleTubeMeansSimpleTube}
Assume that $\AA$ is indecomposable.  Let $\TT_1$ be a simple tube in $\AA$.  Any tube $\TT_2$ in $\TT_1^\perp$ is a simple tube in $\AA$.
\end{proposition}

\begin{proof}
Since $\AA$ is indecomposable, we may assume that there is an unoriented path between $\TT_1$ and $\TT_2$.  By Theorem \ref{theorem:Paths}, we may assume that this path has length at most two.  If there is an oriented path, then the statement follows from Proposition \ref{proposition:OrthogonalAndPathMeansSimple}.  Without loss of generality, we may then assume that there is an indecomposable object $A \in \AA$ with $\Hom(A, \TT_1) \not= 0$ and $\Hom(A, \TT_2) \not= 0$.

Let $Y_1$ and $Y_2$ be the minimal 1-spherical objects associated to the tubes $\TT_1$ and $\TT_2$, respectively.  Seeking a contradiction, assume that $\TT_2$ is not simple.  To ease notation, we denote $T^i_{Y_2} A \in \Db \AA$ by $A_i$. Following Proposition \ref{proposition:Summary}, there is an $i \in \bZ$ such that $A_i \in \AA[0]$ and $A_{i-1}[1] \in \AA[0]$.  Thus, there is a short exact sequence:
$$0 \to A_i \to \Hom(A_i,Y_2)^* \otimes Y_2 \to A_{i-1}[1] \to 0.$$

Since $Y_2 \in Y_1^\perp = {}^\perp Y_1$, we know that $\Hom(A_{i-1}, Y_1) \cong \Hom(A_i, Y_1)$.  This shows that $\Hom(Y_1,A_{i-1}[1]) \cong \Hom(A_{i-1}, Y_1)^* \not= 0$, which contradicts the assumption that $Y_1$ lies in a simple tube (see Proposition \ref{proposition:Summary}).
\end{proof}

\begin{proposition}\label{proposition:SimpleTubesPerpendicular}
Let $\TT_1$ and $\TT_2$ be two distinct simple tubes.  Then $\Hom(\TT_1, \TT_2) = \Ext^1(\TT_1,\TT_2) = 0$.
\end{proposition}

\begin{proof}
Let $X_1$ and $X_2$ be the minimal 1-spherical objects associated to the tubes $\TT_1$ and $\TT_2$, respectively.  Then $X_1$ and $X_2$ are semi-simple and have no nontrivial common direct summands.  This shows that $\Hom(X_1,X_2) = \Hom(X_2,X_1) = 0$.  Using that $\tau X_1 \cong X_1$ and $\tau X_2 \cong X_2$, we see that $\Ext^1(X_1,X_2) \cong \Hom(X_2,X_1)^* = 0$.
\end{proof}

\subsection{\texorpdfstring{$t$-Structures}{t-Structures} induced by tubes}\label{subsection:tstructures}

We recall the definition of a $t$-structure from \cite{BeilinsonBernsteinDeligne82}.

\begin{definition}\label{definition:t}
A \emph{$t$-structure} on a triangulated category $\CC$ is a pair $(D^{\geq 0}, D^{\leq 0})$ of non-zero full subcategories of $\CC$ satisfying the following conditions, where we denote $D^{\leq n} = D^{\leq 0} [-n]$ and $D^{\geq n} = D^{\geq 0} [-n]$:
\begin{enumerate}
\item $D^{\leq 0} \subseteq D^{\leq 1}$ and $D^{\geq 1} \subseteq D^{\geq 0}$
\item $\Hom(D^{\leq 0}, D^{\geq 1}) = 0$
\item $\forall Y \in \CC$, there exists a triangle $X \to Y \to Z \to X[1]$ with $X \in D^{\leq 0}$ and $Z \in D^{\geq 1}$.
\end{enumerate}
Let $D^{[n,m]} = D^{\geq n} \cap D^{\leq m}$ .  We will say the $t$-structure $(D^{\geq 0}, D^{\leq 0})$ is \emph{bounded} if and only if every object of $\CC$ is contained in some $D^{[n,m]}$.
\end{definition}

It is shown in \cite{BeilinsonBernsteinDeligne82} that the \emph{heart} $\HH = D^{\leq 0} \cap D^{\geq 0}$ is an abelian category.  We will use the following proposition from \cite[Proposition 4.2]{BergVanRoosmalen11} (see also \cite[Theorem 1]{Ringel05}).

\begin{proposition}\label{proposition:HereditaryTilting}
Let $\AA$ be an abelian category and let $(D^{\geq 0}, D^{\leq 0})$ be a bounded $t$-structure on $\Db \AA$.  If all the triangles $\tri X Y Z$ with $X \in D^{\leq 0}$ and $Z \in D^{\geq 1}$ split, then $D^{\leq 0} \cap D^{\geq 0} = \HH$ is hereditary and $\Db \AA \cong \Db \HH$ as triangulated categories.
\end{proposition}

We will often work with the hereditary category $\HH(Y)$ given in the following construction.

\begin{construction}\label{construction:Heart}
Let $\AA$ be an indecomposable hereditary category and and let $Y$ be a minimal 1-spherical object in $\AA$.  Assume furthermore that $\AA$ is not equivalent to a tube.

Since $\AA$ is not a tube, we know that there is an indecomposable object $B \in \AA$ such that $\Hom_{\Db \AA}(B[0],Y[l]) \not= 0$, for some $l \in \bZ$.  Up to renaming, we may assume that we have chosen $B \in \Db \AA$ such that $\Hom_{\Db \AA}(B,Y[0]) \not= 0$.

We will consider the following $t$-structure: an object $A \in \Db \AA$ lies in $D^{\leq 0}$ if and only if for each indecomposable direct summand $A'$ of $A$, there is a path from $T_Y^i B$ to $A'$, for $i \ll 0$.  Put differently:
$$\ind D^{\leq 0} = \{A \in \ind \Db \AA \mid \mbox{there is a path from $T_Y^i B$ to $A$, for $i \ll 0$}\}.$$

We then have the following description for $D^{\geq 1}$: an object $A \in \Db \AA$ lies in $D^{\geq 1}$ if and only if for each indecomposable direct summand $A'$ of $A$, there is no path from $T_Y^i B$ to $A'$, for any $i \in \bZ$.

It follows from Lemma \ref{key} that this $t$-structure does not depend on the choice of the indecomposable object $B$.  We denote by $\HH(Y) = D^{\leq 0} \cap D^{\geq 1}$ the heart of this $t$-structure.

In the case that $\AA$ is equivalent to a tube, we will write $\HH(Y)$ for $\AA$.
\end{construction}

\begin{remark}
This $t$-structure was also considered in \cite[\S 4.1]{vanRoosmalen08}.
\end{remark}

\begin{proposition}\label{proposition:SemiSimpleInHeart}
Let $\AA$ be a connected abelian category and let $Y$ be a minimal 1-spherical object in $\AA$.  The category $\HH(Y)$ is hereditary and derived equivalent to $\AA$.  Furthermore, $Y \in \HH$ is semi-simple.
\end{proposition}

\begin{proof}
If $\AA$ is generated by the minimal 1-spherical object $Y$ (thus $\AA$ is a tube), then the statement is easy.  So, assume that $\AA$ is not a single tube.

Since $\AA$ is connected, the $t$-structure from Construction \ref{construction:Heart} is bounded, and it follows from Proposition \ref{proposition:HereditaryTilting} that $\HH$ is hereditary and derived equivalent to $\AA$.  We know that $\Hom(Y[-1],B) \not= 0$ so that  Theorem \ref{theorem:DirectingTubes} shows that there cannot be a path from $B$ to $Y[-1]$.  We conclude that $Y \in \HH$.

Since $T^i_Y B \in \HH$ for all $i \in \bZ$, Proposition \ref{proposition:Summary} shows that $Y$ is semi-simple in $\HH$.
\end{proof}

\begin{remark}
It was shown in \cite[Theorem 4.2]{vanRoosmalen12} that for every tube $\TT$ in $\AA$, there is a $t$-structure in $\Db \AA$ with hereditary heart $\HH$ ($\HH$ is derived equivalent to $\AA$) such that $\TT$ is a simple tube in $\HH$.  The $t$-structure considered in that proof is different from the one we consider in this article (i.e. the one given in Construction \ref{construction:Heart}).
\end{remark}
\section{Hereditary categories without exceptional objects}\label{section:NoExceptionals}

Let $\AA$ be an abelian hereditary category with Serre duality.  We will consider the case where $\AA$ has no exceptional objects.  Our main result will be Theorem \ref{theorem:NoExceptionals}, giving a classification up to derived equivalence.

Since $\AA$ has no exceptional objects, $\AA$ has no nonzero projective or injective objects. Recall from Proposition \ref{proposition:SerreDuality} that $\tau: \Db \AA \to \Db \AA$ then restricts to an autoequivalence $\tau: \AA \to \AA$.

\subsection{Choosing a tilt \texorpdfstring{$\HH$}{H}}\label{subsection:Tilt}

Let $\AA$ be a hereditary category with Serre duality and assume that $\AA$ does not have exceptional objects.  By Theorem \ref{theorem:LowExt}, we know there is a 1-spherical object $X$.  With this 1-spherical object, we associate a $t$-structure  with heart $\HH = \HH(X)$ as in Construction \ref{construction:Heart}.

The heart $\HH$ is a hereditary abelian category, derived equivalent to $\AA$ (see Proposition \ref{proposition:SemiSimpleInHeart}).  We want to show that $\HH$ has a set of sincere orthogonal simple tubes, meaning that there is a set $\SS \subseteq \HH$ of mutually perpendicular simple 1-spherical objects such that for every $E \in \HH$, we have $\Hom(E,\SS) \not= 0$.

The following lemma shows that $\HH$ has ``enough'' 1-spherical objects for our purposes.

\begin{lemma}\label{lemma:SS}
Let $\HH = \HH(X)$ as above.  There is a set $\SS \subset \HH$ (containing $X[0]$) of mutually orthogonal 1-spherical objects, such that for all $A \in \Ob \HH$ either $\Hom(A,\SS) \not= 0$ or $\Hom(\SS,A) \not= 0$.
\end{lemma}

\begin{proof}
If $\HH$ is a tube, then $\HH$ is generated by a single 1-spherical object, and the statement is easy.

Using Theorem \ref{theorem:LowExt}, we infer that the perpendicular category $X[0]^\perp$ in $\HH$ also contains a 1-spherical object.  An easy application of Zorn's Lemma yields that there is a maximal set $\SS$ consisting of orthogonal 1-spherical objects.  Obviously, for this set, we have that either $\Hom(A,\SS) \not= 0$ or $\Hom(\SS,A) \not= 0$.
\end{proof}

\begin{proposition}\label{proposition:SincereWhenNoExceptionals}
Let $\AA$ be an abelian hereditary category with Serre duality and without any exceptional objects.  There is a $t$-structure with hereditary heart $\HH$, derived equivalent to $\AA$, and a set $\SS$ of indecomposable 1-spherical objects such that
\begin{enumerate}
\item $\Hom(S_1, S_2) = 0 = \Ext^1(S_1,S_2)$, for every two nonisomorphic $S_1,S_2 \in \SS$
\item each $S \in \SS$ is simple in $\HH$, and
\item for all $A \in \HH$, there is an $S \in \SS$ such that $\Hom(A, S) \not= 0$.
\end{enumerate}
\end{proposition}

\begin{proof}
If $\AA$ is generated by a single 1-spherical object $X$ (thus $\AA$ is a tube), then the statement is trivial.  Thus assume that $\AA$ is not generated by $X$.

Let $\HH$ and $\SS$ be as above, thus $X \in \AA$ is any 1-spherical object whose existence is given by Theorem \ref{theorem:LowExt}, $\SS$ is given by Lemma \ref{lemma:SS}, and $\HH$ is given in Proposition \ref{proposition:SemiSimpleInHeart}.  The first statement has been shown in Lemma \ref{lemma:SS}.

Next, we show that all objects in $\SS$ are simple.  It follows from Proposition \ref{proposition:SemiSimpleInHeart} that $X[0]$ is simple in $\HH$.  Let $S \in \SS$.  In particular $S \in \HH \subseteq \DD^{\leq 0}$ and we know that there is a path $B=B_0 \to B_1 \ldots \to B_n \to S$ in $\HH$ where $\Hom(B_0,X) \not= 0$ and we may assume that $B_0$ and $B_n$ do not lie in the tubes containing $X$ and $S$ respectively.  According to Proposition \ref{proposition:Summary}, it suffices to show that $T^i_S B_n \in \HH$ for all $i \in \bZ$.

By Lemma \ref{lemma:ClosedUnderTwist}, we know that $T_S^i B_0 \in \HH$ for all $i \geq 0$.  We thus need only to consider the case where $i < 0$.  Seeking a contradiction, let $i$ be the largest integer such that $T_S^i B_0 \not\in \HH$, thus Lemma \ref{lemma:ClosedUnderTwist} yields that $T_S^i B_0[1] \in \HH$.  Note that there is a path $T^i_S B_0 \to T^i_S B_1 \ldots \to T^i_S B_n \to S$ in $\Db \AA$.  Applying $\Hom_{\Db \AA}(-,X[0])$ to the triangle
$$\RHom(S,T_S^i B_0, S) \otimes S \to T_S^i B_0 \to T_S^{i+1} B_0 \to \RHom(S,T_S^i B_0, S) \otimes S[1]$$
we find that $\Hom(T_S^i B_0, X[0]) \not= 0$ (this uses that $\Hom(B_0,X[0]) \not= 0$ and that $S \in X[0]^\perp = {}^\perp X[0]$).  Lemma \ref{lemma:Ringel} implies that there is a path from $X$ to $T_S^i B_0[1]$.  Since $T_S^i B_0[1]$ does not lie in the tube containing $X[0]$, this implies that there is an indecomposable $Y \in \HH$, which also does not lie in the tube containing $X[0]$, such that $\Hom(X[0], Y) \not= 0$.  However, the object $X[0]$ is simple in $\HH$, and the existence of the objects $B_0, Y \in \HH$ contradicts Proposition \ref{proposition:Summary}.  We have obtained the required contradiction and may conclude that $S$ is simple in $\HH$.

It follows from Lemma \ref{lemma:SS} that $\Hom(\SS,A) \not= 0$ or $\Hom(A, \SS) \not= 0$, for each $A \in \HH$.  If $A$ lies in a tube generated by $\SS$, then $\Hom(A, \SS) \not= 0$.  We will now show $\Hom(A, \SS) \not= 0$ when $A$ does not lie in such a tube.

Let $S \in \SS$ and let $\TT$ be the tube generated by $S$.  We will assume that $S \not\cong X$.  By the construction of $\HH$, we know that there is an indecomposable object $B \in \HH$ and a path from $T^{i}_X B \in \HH$ to $S$ for $i \ll 0$.  In particular, there is an indecomposable object $A'$ (not contained in the tube $\TT$ generated by $S$) such that $\Hom(A',S) \not= 0$.  Since $S$ is simple, Proposition \ref{proposition:Summary} yields that $\Hom(S,A) = 0$ for all indecomposables $A \not\in \TT$.

This completes the proof.
\end{proof}

\begin{remark}\label{remark:TorsionIsFiniteLength}
Every object lying in one of the tubes generated by $\SS$ has finite length, while all other objects have infinite length.  Indeed, for each $A$ not lying in one of the tubes of $\SS$, there is a short exact sequence
$$0 \to T^*_S A \to A \to \Hom(A,S)^* \otimes S \to 0$$
in $\HH$ for some $S \in \SS$.  This shows that $A$ has infinite length.
\end{remark}

We may define a torsion theory on $\HH$ in the standard way:
\begin{eqnarray*}
\TT &=& \{A \in \Ob \HH \mid \mbox{$A$ has finite length}\} \\
\FF &=& \{A \in \Ob \HH \mid \mbox{every nonzero direct summand of $A$ has infinite length}\}.
\end{eqnarray*}
Note that $\tau \TT = \TT$ and thus also $\tau \FF = \FF$.  It then follows from Serre duality that $\Ext^1(\FF,\TT) = 0$ such that the objects of $\TT$ act as injectives with respect to $\FF$.  In particular, every object in $\HH$ is a direct sum of a torsion and a torsion-free object.

\begin{proposition}\label{proposition:WhenTorsion}
Let $(\TT, \FF)$ be the aforementioned torsion theory on $\HH$.  An object $A \in \HH$ is torsion if and only if $\chi(A,S) = 0$ for all $S \in \SS$.
\end{proposition}

\begin{proof}
Assume that $A$ is torsion.  Without loss of generality, we may assume that $A$ is indecomposable.  By Remark \ref{remark:TorsionIsFiniteLength} we know that $A$ lies in a tube $\TT$ generated by some $S \in \SS$.  By Proposition \ref{proposition:SphericalsInTubes}, the tube is equivalent to the category of nilpotent representations of the one-loop quiver (thus the category of nilpotent representations of $k[t]$) and thus $\chi(A,S) = 0$.

For any $S' \in \SS$, nonisomorphic to $S$, we know that $\Hom(S',A) = 0$ and $\Hom(A,S') \cong \Ext^1(S',A)^*=0$ and thus also $\chi(S',A) = 0$.

Assume now that $A \in \HH$ is torsion-free.  Since $\SS$ is torsion, we know that $\Ext^1 (A,S) = 0$ so that $\chi(A,S) = \dim \Hom(A,S)$.  The statement then follows from Proposition \ref{proposition:SincereWhenNoExceptionals}.
\end{proof}

\subsection{Maps from torsion-free objects to torsion objects}

In this section, we will show that the category $\HH$ admits a torsion-free object $F$ such that $\dim_k \Hom(F,S)$ is bounded when $S$ ranges over $\SS$.  Our first result is Proposition \ref{proposition:BoundForTwo} where we give sufficient condition on this object $F$ to ensure the existence of such a bound.  Lemmas \ref{lemma:AtLeastTwo} and \ref{lemma:EllipticLemma} prove that such objects $F$ can be found in $\FF$.

For the following proposition, note that we do not require the object $F$ to be endo-simple.

\begin{proposition}\label{proposition:BoundForTwo}
Assume that $\HH$ has a torsion-free object $F \in \FF$ with $e=\dim \Ext^1(F,F) \geq 2$ and such that every nonzero $F \to \t F$ is a monomorphism.  For every $S \in \SS$, we have $\dim \Hom(F,S) \leq 2e-2$.
\end{proposition}

\begin{proof}
Consider the short exact sequence $0 \to F \stackrel{f}{\rightarrow} \t F \to C \to 0$ based on a nonzero map $f: F \to \t F$ (thus $C \cong \coker f$).  For any $S \in \SS$ we have
$$\chi(C, S) = \chi(\t X, S) - \chi(X, S) = 0$$
where we have used that $S \cong \t S$.  Proposition \ref{proposition:WhenTorsion} now yields that $C$ is torsion.

We will show that, for any $S \in \SS$, there is an $f: F \to \t F$ such that $\Hom(C, S) \not= 0$.  Since $C$ is torsion, this would imply that $S$ is also a subobject of $C$.  Using furthermore that $\Ext^1(F,C) = 0$ (since $C$ is torsion and $F$ is torsion-free), we obtain
\begin{align*}
\dim \Hom(F,S) &\leq \dim \Hom(F,C) \\
&= \chi(F,C) \\
&= \chi(F, \t F) - \chi(F,F) \\
&= -2\chi(F,F) \\
&= 2e-2 \dim \Hom(F,F) \\
& \leq 2e-2.
\end{align*}

We fix an $S \in \SS$, and we look for nonzero morphisms $f: F \to \t F$ and $g: \t F \to \t S$ such that $g \circ f = 0$.  This would imply that $\Hom(C,S) \not= 0$ (where $C \cong \coker f$).

To this end, we choose an isomorphism $\alpha: \t S \to S$, and we define a linear map
\begin{eqnarray*}
\varphi: \Hom(F,\tau F) &\to& \End_k(\Hom(F,S)) \\
f &\mapsto& \varphi_f 
\end{eqnarray*}
where $\varphi_f(g) = \alpha \circ \t g \circ f$, thus
$$\varphi_f (g) : F \stackrel{f}{\rightarrow} \t F \stackrel{\t g}{\rightarrow} \t S \stackrel{\alpha}{\rightarrow} S.$$

If $\varphi: f \to \varphi_f$ is not a monomorphism, then we are done.  Indeed, assume that $f$ is a nonzero element in $\Hom(F, \tau F)$ such that $\varphi_f = \varphi_0$, thus $\varphi_f(h) = 0$, for all $h: F \to S$.  Then we find $\t h \circ f = 0$ as requested.

Thus, assume that $\varphi: \Hom(F,\t F) \to \End(\Hom(F,S))$ is a monomorphism.  We will show that $\varphi_f$ is not an isomorphism for some $f \in \Hom(F,\t F) \setminus \{0\}$.   In the vector space $\End(\Hom(F,S))$, the set of non-invertible endomorphisms forms a hypersurface $H$ and $\im \varphi$ forms a subspace of $\End(\Hom(F,S))$ of dimension $\dim \Ext^1(F,F) \geq 2$.  In the associated projective space $\mathbb{P}(\End(\Hom(F,S))$, B\'ezout's theorem implies that the hypersurface $\mathbb{P}(H)$ intersects nontrivially with $\mathbb{P}(\im \varphi)$.  We infer that $\im \varphi$ contains non-invertible, and hence non-injective, nonzero endomorphisms of $\Hom(F,S)$.

In particular, let $f \in \Hom(F,\t F) \setminus \{0\}$ be such that $\varphi_f$ is not invertible.  Then there is a nonzero $h \in \Hom(F,S)$ such that $\varphi_f(h)=0$, and thus $\t h \circ f = 0$ as requested.
\end{proof}

The following lemma shows the existence of an object $F$ satisfying the conditions of Proposition \ref{proposition:BoundForTwo}, under the additional minimality condition on $\dim \Ext^1(F,F)$.

\begin{lemma}\label{lemma:AtLeastTwo}
Let $F \in \HH$ be a torsion-free object.  If $d=\dim \Ext^1(F,F)$ is minimal among all torsion-free objects and $d \geq 2$, then every nonzero map $F \to \t F$ is a monomorphism.
\end{lemma}

\begin{proof}
The proof resembles that of Lemma \ref{MonoLemma}.  First note that by Proposition \ref{extraExtension}, we may assume that $F$ is endo-simple.  Let $f: F \to \t F$ be nonzero with kernel $K$, cokernel $C$, and image $I$.  Since $K$ and $I$ are subobjects of torsion-free objects, they are torsion-free as well.  We have the following exact sequences
$$0 \to K \to F \to I \to 0, \qquad 0 \to I \to \t F \to C \to 0.$$
By applying $\Hom(F,-)$ to the first sequence, and $\Hom(-,\t F)$ to the second, we find
\begin{eqnarray*}
\dim \Ext^1(F, K) &=& d-1+\chi(F, I) \\
\dim \Ext^1(C, \t F) &=& d-1 - \chi(F, I)
\end{eqnarray*}
Since $K$ is a subobject of $F$ and $\Ext^1(-,K)$ is right exact, we find $\dim \Ext^1(F,K) \geq \dim \Ext^1(K,K)$.  Likewise, we find $\dim \Ext^1(C, \t F) \geq \dim \Ext^1(C,C)$.  Combining these inequalities with the two equations above, we find that either $\dim \Ext^1(K,K) < d$ or $\dim \Ext^1(C,C) < d$.

In the former case, the minimality of $d$ implies that $K$ is zero, and hence $f$ is a monomorphism.  In the latter case, $C$ is necessarily torsion (possibly zero) such that $\chi(C,S) = 0$ for all $S \in \SS$ by Proposition \ref{proposition:WhenTorsion}.  From this, we find that $\chi(I,S) = \chi(\t F,S)$ (by the second short exact sequence) and that $\chi(\t F,S) = \chi(F,S)$ (since $\t S \cong S$).  We see that $\chi(I,S) = \chi(F,S)$, and hence $\chi(K,S)=0$ for all $S \in \SS$.  Since $K$ is torsion-free, Proposition \ref{proposition:WhenTorsion} implies that $K \cong 0$.  We have shown that $f$ is a monomorphism.
\end{proof}

As mentioned before, the case where $d=1$ in Lemma \ref{lemma:AtLeastTwo} will not allow us to apply Proposition \ref{proposition:BoundForTwo} directly; this case merits some extra attention.  It will follow from the classification in Theorem \ref{theorem:NoExceptionals} below that $\Db \HH \cong \Db \coh \bX$ for an elliptic curve $\bX$.  For now, we will have to consider the possibility that there is an indecomposable torsion-free object $F$ such that $\t F \not\cong F$.

\begin{lemma}\label{lemma:EllipticLemma}
Assume that there is a torsion-free object $G$ with $d = \dim \Ext^1(G,G) = 1$.  If there is an object $A$ such that $\t A \not\cong A$, then there is an indecomposable torsion-free object $F$ such that $\t F \not\cong F$, every nonzero morphism $F \to \t F$ is a monomorphism, and $\dim \Ext^1(F,F) \geq 2$.
\end{lemma}

\begin{proof}
Let $F$ be an indecomposable torsion-free module such that $\t F \not\cong F$ (thus $F$ does not lie in a tube of $\HH$).  We will choose $F$ such that $e=\dim \Ext^1(F,F)$ is minimal among all the objects satisfying these properties.  Note that we do not require $F$ to be endo-simple.

We start by showing that $\dim \Ext^1(F,F) \geq 2$.  Seeking a contradiction, assume that $\dim \Ext^1(F,F) = 1$.  Proposition \ref{extraExtension} implies that $F$ is endo-simple (recall that $\HH$ does not have exceptional objects).  Lemma \ref{lemma:SmallD}, together with $F \not\cong \t F$, now shows that $\HH$ has an exceptional object.  This contradiction shows that $\dim \Ext^1(F,F) \geq 2$.

Let $f: F \to \t F$ be a nonzero morphism.  Again, we associate the following short exact sequences:
$$0 \to K \to F \to I \to 0, \qquad 0 \to I \to \t F \to C \to 0,$$
where $K \cong \ker f$, $I \cong \im f$, and $C \cong \coker f$.  If $f$ were an epimorphism (and hence $I \cong \t F$), then we can use $\t T \cong T$ (for all $T \in \TT$) to see that $\chi(K,\TT) = 0$.  Since $K$ is torsion-free, we know that $\chi(K,\TT) = 0$ implies that $K \cong 0$ (Proposition \ref{proposition:WhenTorsion}), contradicting $\t F \not\cong F$.

Thus, assume that $f$ is not an epimorphism and thus $C \not= 0$.  The short exact sequences above yield the following equalities:
\begin{eqnarray*}
\chi(F, K) - \chi(F,F) + \chi(F,I) &=& 0 \\
\chi(I, \t F) - \chi(\t F,\t F) + \chi(C, \t F) &=& 0.
\end{eqnarray*}
Using that $\dim \Hom(F,F) > \dim \Hom(F,K)$ and $\dim \Hom(\t F, \t F) > \dim \Hom(C, \t F)$, we find
\begin{eqnarray*}
\dim \Ext^1(F, K) &<& e +\chi(F, I) \\
\dim \Ext^1(C, \t F) = \dim \Hom(F,C) &<& e - \chi(F, I)
\end{eqnarray*}
We recall that $\HH$ is hereditary and thus that $\Ext^1(-,K)$ and $\Ext^1(C,-)$ are right exact.  Since $\dim \Ext^1(F,K) \geq \dim \Ext^1(K,K)$ and $\dim \Ext^1(C, \t F) \geq \dim \Ext^1(C,C)$, the above inequalities show that either $\dim \Ext^1(K,K) < e$ or $\dim \Ext^1(C,C) < e$.

We will assume the latter; the former is similar.  Due to the minimality of $e$, we know that $C \cong \t C$ and hence $\t$ permutes the indecomposable direct summands of $C$.  Theorem \ref{theorem:TubeCriterium} shows that every indecomposable direct summand of $C$ lies in a tube and since $\HH$ has no exceptional objects, we know that every tube is homogeneous.

Consider a direct sum decomposition $C \cong \oplus_i C_i$ into indecomposable objects.  By Theorem \ref{theorem:DirectingTubes}, we know that we can choose the labeling such that $\chi(C_j, C_i) \leq 0$ when $i \leq j$.  In particular, since $\chi(C_1,C_1) = 0$, we have $\chi(C,C_1) \leq 0$.  Also note that there is a path from every direct summand of $K$ to $C_i$.  Theorem \ref{theorem:DirectingTubes} shows that there is no path from $C_i$ to any direct summand of $K$ and Lemma \ref{lemma:Ringel} yields that $\Ext^1(K,C_i)=0$.  Hence, $\chi(K,C_i) \geq 0$.  From the above short exact sequences, we find
\begin{eqnarray*}
\chi(I,C_1) - \chi(F,C_1) + \chi(K,C_1) &=& 0 \\
\chi(I,C_1) - \chi(\t F,C_1) + \chi(C,C_1) &=& 0.
\end{eqnarray*}
Using that $\chi(\t F,C_1) = \chi(F,\t^{-1} C_1) = \chi(F,C_1)$, we can take the difference of these two equations and find that $\chi(K,C_1) - \chi(C,C_1) = 0$.  The signs of these terms now imply that
$$\chi(K,C_1) = \chi(C,C_1) = 0.$$

In particular, for all $j > 1$ we have that either either $C_j \cong C_1$ or $C_j \in C_1^\perp$, and we have that $\chi(C,C_2) \leq 0$.  We can continue this procedure to conclude that $C_i \in C_j^\perp$, for all $C_i \not\cong C_j$.  Moreover, we find that $\chi(K,C) = 0 = \chi(C,K)$.

Following Lemma \ref{lemma:SS}, we can extend the set $\{C_i\}_i$ to a set $\SS' \subseteq \HH$ of 1-spherical objects such that either $\Hom(A,\SS') \not= 0$ or $\Hom(\SS',A) \not= 0$ for any $A \in \HH$.  We claim that $\chi(K_j,S') = 0$ for all indecomposable direct summands $K_j$ of $K$ and all $S' \in \SS'$, so that it follows from Theorem \ref{theorem:DirectingTubes} that $K$ lies in the subcategory generated by $\SS'$.

For an $S' \in \SS'$, we easily find that $\chi(I,S') = \chi(\t F, S')$ and, using that $\chi(\t F, S') = \chi(F, \t^{-1} S') = \chi(F,S')$, thus also that $\chi(K,S') = 0$.

Seeking a contradiction, let $K_j$ be an indecomposable direct summand of $K$ such that $\chi(K_j, S') < 0$.  This would imply that $\dim \Ext^1(K_j, S') \not= 0$, and thus by Lemma \ref{lemma:Ringel} there is a path from $S'$ to $K_j$.  We can then concatenate this path with the given path from $K_j$ to $C_1$ giving a path from $S'$ to $C_1$ which passes $F$.  Proposition \ref{proposition:OrthogonalAndPathMeansSimple} then yields that the tubes containing $S'$ and $C_1$ are simple tubes.  However, this would imply that $S' \in \TT$, and we have established that there are no paths from $\TT$ to $\FF$ (recall that $F \in \FF$).  We may conclude that $\chi(K_j, S') \geq 0$ for all $S' \in \SS'$ and all indecomposable direct summands $K_j$ of $K$.  

It now follows from $\chi(K,S') = \sum_j \chi(K_j,S') = 0$ and $\chi(K_j, S') \geq 0$ that $\chi(K_j,S') = 0$.  Thus, by Proposition \ref{proposition:WhenTorsion}, we know that $K$ lies in the subcategory generated by $\SS'$.  In particular, $\t K \cong K$.  We know that there are paths from all indecomposable direct summands of $K$ to all direct summands of $C$, so that if $K$ is nonzero then Proposition \ref{proposition:OrthogonalAndPathMeansSimple} shows that $K,C \in \TT$.  Since $K$ is a subobject of $F$, this implies that $K \cong 0$, and we can conclude that $f$ is a monomorphism.
\end{proof}

\subsection{The quotient category \texorpdfstring{$\HH / \TT$}{H / T}}\label{subsection:Quotient}

In this subsection, we will describe the quotient category $\HH / \TT$.  We will use the notation introduced before.  Thus in particular, $\HH$ is a connected hereditary category with Serre duality without exceptional objects.  Let $\SS$ be the collection of simple objects and assume that every object in $\HH$ maps nonzero to at least one object in $\SS$ (see Proposition \ref{proposition:SincereWhenNoExceptionals} and Remark \ref{remark:TorsionIsFiniteLength}).  Let $\TT$ be the Serre subcategory generated by $\SS$, and let $\FF$ be the right $\Hom(-,-)$ orthogonal of $\TT$.  Note that by Serre duality, we have $\Ext^1(\FF,\TT) = 0$ so that every object of $\TT$ behaves as an injective object with respect to $\FF$.

\begin{lemma}\label{lemma:QuotientSemiSimple}
$\HH/ \TT$ is semi-simple in the sense that $\Ext^1(-,-)$ between any two objects is zero.
\end{lemma}

\begin{proof}
As in \cite[Corollary IV.1.4]{ReVdB02}.  We will repeat the argument for the benefit of the reader.  It is sufficient to show that for any $A,B \in \HH$ there is a subobject $B'$ of $B$ with $B / B' \in \TT$ such that $\Ext^1(B',A) = 0$.  Since every object in $\HH$ is a direct sum of an object in $\FF$ and an object in $\TT$, we may assume that $B \in \FF$.

Choose a $B' \subseteq B$ (with $B / B' \in \TT$) such that $\dim_k \Ext^1(B',A)$ is minimal.  Seeking a contradiction, we assume that $\dim_k \Ext^1(B',A) \not= 0$.

Since $\HH$ has no nonzero injective objects, we have $\Ext^1(B',A) \cong \Hom(\t^{-1} A, B')^*$.  Let $f: \t^{-1} A \to B'$ be any nonzero morphism.  Let $T$ be the torsion subobject of $B' / \im f$; since $\Ext^1(\FF, \TT) = 0$, $T$ is a direct summand of $B' / \im f$.  We have the following commutative diagram with exact rows and columns:
$$\xymatrix{
& & 0\ar[d] & 0\ar[d] \\
0 \ar[r] & \im f \ar[r] \ar@{=}[d]& B'' \ar[r] \ar[d] & {B''/ \im f} \ar[r] \ar[d] & 0 \\
0 \ar[r] & \im f \ar[r] & B' \ar[r] \ar[d] & {B'/ \im f} \ar[r] \ar[d] & 0 \\
& & T \ar@{=}[r] \ar[d] & T \ar[d] \\
& & 0 & 0
}$$

Note that $B'' \subseteq B' \subseteq B$ (hence $B'' \in \FF$) and that $B / B'' \in \TT$.  Furthermore, $B'' / \im f$ is torsion-free and, by the minimality of $\dim_k \Ext^1(B',A)$, we know that $\dim_k \Ext^1(B',A) = \dim_k \Ext^1(B'',A)$.

It follows from Proposition \ref{proposition:SincereWhenNoExceptionals} that there is a simple $S \in \TT$ such that $\Hom(\im f, S) \not= 0$.  Consider the short exact sequence
$$0 \to T_S^*(B'') \to B'' \to \Hom(B'',S)^* \otimes S \to 0.$$
Since $B'' / \im f \in \FF$ and $\Ext^1(\FF, \TT) = 0$, we know that any morphism $\im f \to S$ factors through the inclusion $\im f \to B''$ and thus as
$$\im f \to B'' \to \Hom(B'',S)^* \otimes S \to S.$$
Since $\Hom(\im f, S) \not= 0$, we know that $\im f \not\subseteq T_S^*(B'')$.  We may conclude that the monomorphism $\Hom(\t^{-1} A, T_S^*(B'')) \to \Hom(\t^{-1} A, B'')$ is not an epimorphism (since the image does not contain the map $f$) and thus
$$\dim \Hom(\t^{-1} A, T_S^*(B'')) \lneqq \dim \Hom(\t^{-1} A, B'') = \dim \Hom(\t^{-1} A, B').$$

However, it is easily checked that $B / T^*_S(B'') \in \TT$.  This contradicts the minimality of $\dim_k \Ext^1(B',A)$, and hence we know that $\Ext^1(B',A) = 0$.  This concludes the proof.
\end{proof}

We will write $\pi: \HH \to \HH / \TT$ for the quotient functor.  Let $F$ be an object in $\HH/ \TT$ and let $\tilde{F}$ be a lift of $F$ in $\HH$, thus $\pi(\tilde{F}) \cong F$. Since every object in $\HH$ is a direct sum of an object in $\TT$ and an object in $\FF$, we can choose the lift $\tilde{F}$ to be in $\FF$.  We put
$$v(F) = (\dim \Hom(\tilde{F}, S))_{S \in \SS}.$$
For lifts $\tilde{F},\tilde{G}$ of $F,G$ we have
$$\Hom_{\HH}(\tilde{F},\tilde{G}) \hookrightarrow \Hom_{\HH/ \TT} (F,G)$$
since the image of a map $\tilde{F} \to \t \tilde{G}$ is torsion-free.

Furthermore, for an object $F \in \HH / \TT$, we will write $\t F$ for $\pi (\t \tilde{F})$.  Since $\t \TT = \TT$, the object $\t F \in \HH / \TT$ is well-defined.

\begin{lemma}
The function $v$ is well-defined and additive on $\HH/ \TT$. Furthermore $v(\t F) = v(F)$. If $F \not\cong 0$ then $v(F) \not= 0$.
\end{lemma}

\begin{proof}
To show that $v$ is well-defined, let $\tilde{F}, \bar{F}$ be lifts of an object $F \in \HH/\TT$.  Consider an isomorphism $\pi(\tilde{F}) \cong \pi(\bar{F})$.  Using that $\bar{F}$ has no nonzero torsion subobjects, the existence of the isomorphism means that there is a subobject $G$ of $\tilde{F}$ such that $\tilde{F}/G$ is torsion, and a map $f: G \to \bar{F}$ such that both the kernel $\ker f$ and the cokernel $\coker f$ are torsion.  Using that $\chi(\TT,S) = 0$ for all $S \in \SS$, we find that
$$\dim \Hom(\tilde{F}, S) = \dim \Hom(G, S) = \dim \Hom(\bar{F}, S),$$
as required.

The equality $v(\t F) = v(F)$ follows from
$$\Hom(\tilde{F}, S) \cong \dim \Hom(\tilde{F}, \t^{-1} S) \cong \dim \Hom(\t \tilde{F},S),$$
where we have used that $\t^{-1} S \cong S$.

The last statement follows from Proposition \ref{proposition:SincereWhenNoExceptionals}. 
\end{proof}

\begin{lemma}\label{lemma:QuotientHasSimple}
Assume that there is an $E \in \HH/ \TT$ such that $v(E)$ is bounded.  Then $\HH/ \TT$ contains a simple object.
\end{lemma}

\begin{proof}
Choose $E$ as in the statement of the lemma such that $\max_{S} v_S(E)$ is minimal. We assume, in addition, that $\tilde{E}$ is indecomposable.

Assume that $E \in \HH / \TT$ is not simple, i.e. $E = E_1 \oplus E_2$ in $\HH / \TT$. Then there are lifts $\tilde{E}, \tilde{E_1}, \tilde{E_2}$ and an exact sequence (see \cite[Corollaire 3.1.1]{Gabriel62})
$$0 \to \tilde{E_1} \to \tilde{E} \to \tilde{E_2} \to 0$$
which is not split since $\tilde{E}$ is indecomposable.  Hence $\Hom_\HH(\tilde{E_1}, \t \tilde{E_2}) \not= 0$ and thus
$\Hom_{\HH / \TT}(E_1, \t E_2) \not= 0$.

Choose a nonzero map $E_1 \to \t E_2$ and let $F$ be its image. By Lemma \ref{lemma:QuotientSemiSimple}, $F$ is a summand of both $E_1$ and $\t E_2$. For an $S \in \SS$ such that $v_S(F) \not= 0$, we have
$$v_S(F) \leq \min(v_S(E_1), v_S(\t E_2)) \leq \min(v_S(E_1), v_S(E_2)) < v_S(E_1) + v_S(E_2) = v_S(E)$$
where we have used that $v_S(F) \not= 0$ and thus also $v_S(E_1) \not= 0 \not= v_S(E_2)$. This contradicts the minimality of $E$.
\end{proof}

\subsection{Proof of classification}  We are now ready to prove Theorem \ref{theorem:NoExceptionals} below.  Let $\HH$ be the abelian category derived equivalent to $\AA$, given in \S\ref{subsection:Tilt}.

\begin{proposition}\label{proposition:NoExceptionals}
Assume that $\HH/ \TT$ contains a simple object.  Then $\Db \HH \cong \Db \coh \bX$ for a smooth projective curve $\bX$.
\end{proposition}

\begin{proof}
Let $\NN$ be the full subcategory of noetherian objects in $\HH$.   Since $\HH$ has no nonzero exceptional objects, neither does $\NN$.  In particular, $\HH$ does not have nonzero projective and injective objects and the Auslander-Reiten translate $\t: \Db \HH \to \Db \HH$ restricts to an autoequivalence $\t: \HH \to \HH$.  We have that $\t \NN = \NN$.

We claim that $\NN$ contains at least one indecomposable non-simple object, namely the object which becomes simple in $\HH/ \TT$ (see Lemma \ref{lemma:QuotientHasSimple}). Indeed, let $N \in \HH/\TT$ be a simple object and let $\tilde{N} \in \FF$ be a lift, thus $\pi\tilde{N} \cong N$.  Let $M_0 \subseteq M_1 \subseteq M_2 \subseteq \ldots$ be an increasing sequence of subobjects of $\tilde{N}$.  Since $N$ is simple in $\HH / \TT$ and $M_i \not\in \TT$, we know that $\pi(M_i) \cong N$.  Hence $\tilde{N} / M_i \in \TT$.  Since there are epimorphisms $\tilde{N} / M_i \to \tilde{N} / M_j$ for all $i \leq j$, and $\TT$ is a length category, we know that these epimorphisms become isomorphisms for $i,j \gg 0$.  This shows that the sequence $(M_i)_i$ of subobjects of $\tilde{N}$ stabilizes.  We have shown that $\tilde{N}$ is a noetherian object.

Let $\NN'$ be a connected component of $\NN$.  Clearly $\t \NN' = \NN'$.  It follows from the classification in \cite{ReVdB02} that $\NN' \cong \coh \bX$ where $\bX$ is a smooth projective curve.

We now have a fully faithful map $\coh \bX \to \HH$ which, due to Serre duality, also preserves $\Ext^1(-,-)$. It follows that the derived functor $F: \Db \coh X \to \Db \HH$ is fully faithful as well.  Since $\Db \coh \bX$ is saturated this map yields a semi-orthogonal decomposition of $\Db \HH$, and since $\Db \coh X$ is closed under the Serre functor this semi-orthogonal decomposition is a genuine decomposition (see Corollary \ref{corollary:SemiOrthogonal}). By the fact that $\Db \HH$ is indecomposable, we obtain $\Db \coh X = \Db \HH$.
\end{proof}

\theoremNoExceptionals*

\begin{proof}
Let $\SS$ be as in Proposition \ref{proposition:SincereWhenNoExceptionals}.  Let $\TT$ be the Serre subcategory generated by $\SS$.

If $\AA / \TT$ is zero, then $\AA \cong \TT$, and hence $\AA$ is a homogeneous tube.

Thus, assume that $\AA / \TT$ is nonzero.  In this case, let $E \in \FF$ be a nonzero torsion-free object in $\AA$ such that $d = \dim \Ext^1(E,E)$ is minimal among all nonzero torsion-free objects.

If $d \geq 2$, then it follows from Lemma \ref{lemma:AtLeastTwo} and Proposition \ref{proposition:BoundForTwo} that $v(E)$ is bounded, and hence $\AA/ \TT$ contains a simple object by Lemma \ref{lemma:QuotientHasSimple}.  The classification follows from Proposition \ref{proposition:NoExceptionals}.

If $d=1$, we consider two cases.  The first case is where $F \cong \t F$ for all indecomposable torsion-free objects.  In this case, it follows from \cite{vanRoosmalen08} that $\Db \AA \cong \Db \coh \bX$ where $\bX$ is an elliptic curve (the proof in \cite{vanRoosmalen08} only uses that $\bS X \cong X[1]$ for all $X \in \Db \AA$, not that $\bS \cong [1]$).

The second case we consider is where there is an object $F \in \FF$ such that $\t F \not\cong F$.  We can then use Lemma \ref{lemma:EllipticLemma} and Proposition \ref{proposition:BoundForTwo} to see that there is an object $F \in \FF$ such that $v(F)$ is bounded, and hence $\AA/ \TT$ contains a simple object by Lemma \ref{lemma:QuotientHasSimple}.  The classification follows from Proposition \ref{proposition:NoExceptionals}.
\end{proof}

\begin{remark}
It follows from the classification in Theorem \ref{theorem:NoExceptionals} that the final case in the proof cannot occur.  In particular, there are no categories satisfying the conditions in Lemma \ref{lemma:EllipticLemma}.
\end{remark}

\begin{remark}
As a corollary to Theorem \ref{theorem:NoExceptionals}, we can recover the classification of abelian 1-Calabi-Yau categories from \cite{vanRoosmalen08}.  However, that classification was used in the proof of Theorem \ref{theorem:NoExceptionals}.
\end{remark}

\begin{remark}\label{remark:WhatIsHH}
It now follows from the classification that the category $\HH$ we constructed in \S\ref{subsection:Tilt} is equivalent to either a homogeneous tube, or to the category $\coh \bX$ of coherent sheaves on a smooth projective curve $\bX$ of genus at least one.
\end{remark}

\begin{corollary}\label{corollary:PerpendicularToSimpleTubeNoExceptionals}
Let $\AA$ be as in Theorem \ref{theorem:NoExceptionals}.  Let $S$ be a 1-spherical object.  The category $S^\perp$ is a direct sum of homogeneous tubes.
\end{corollary}

\begin{proof}
We only need to consider the case where $\Db \AA \cong \Db \coh \bX$ for a smooth projective curve $\bX$.  If the genus of $\bX$ is at least 2, then the only 1-spherical objects in $\Db \coh \bX$ correspond to simple sheaves in $\coh \bX$, and the statement follows easily.

When the genus is 1, then $\bX$ is an elliptic curve and the statement follows from the classification of coherent sheaves on $\coh \bX$ (\cite{Atiyah57}, see also \cite{BruningBurban07, vanRoosmalen08}).
\end{proof}

\begin{corollary}\label{corollary:OnesimpleMeansEnoughSimples}
Let $\AA$ be as in Theorem \ref{theorem:NoExceptionals}. Let $\SS$ be the set of all simple 1-spherical objects.  If $\SS$ is nonempty, then $\SS^\perp = {}^\perp \SS = 0$. 
\end{corollary}

\begin{proof}
Let $S$ be a simple 1-spherical object in $\AA$.  In each of the examples, we know that $S^\perp$ is a direct sum of tubes.  It follows from Proposition \ref{proposition:OrthogonalToSimpleTubeMeansSimpleTube} that each of those tubes is simple in $\AA$.  We then obtain $\SS^\perp = {}^\perp \SS = 0$ as requested.
\end{proof}
\section{Numerically finite categories}

Let $\AA$ be an indecomposable abelian hereditary category with Serre duality.  If $\AA$ has no exceptional objects, then $\AA$ is derived equivalent to one of the categories in Theorem \ref{theorem:NoExceptionals}.  We will now consider the case where $\AA$ may have exceptional objects, but where every exceptional sequence is finite.  This is the case, for example, when $\AA$ is numerically finite.

Let $\EE = (E_i)_{i=1, \ldots, n}$ be a (finite) maximal exceptional sequence in $\Db \AA$.  It follows from Proposition \ref{proposition:SequenceMeansObject} that we may assume that $\EE \subseteq \AA$ and that $\EE$ is a strong exceptional sequence, thus $E = \oplus_{i=1}^n E_i$ is a partial tilting object in $\AA$.  Throughout, we will assume that $\EE$ is such an exceptional sequence.

We will split our discussion into two cases.  The first case is where $\AA$ has a tilting object; these categories are understood by the classification in \cite{Happel01} (see Theorem \ref{theorem:Happel}).  In the second case, we will show that we have ``enough'' generalized 1-spherical objects (as in Proposition \ref{proposition:SincereWhenNoExceptionals}) and use this to find a derived equivalent category $\HH$ which we can show to be noetherian.  The categories are then understood via the classification in \cite{ReVdB02}.

Since the first case is easily dealt with, we will assume that $\AA$ does not have a tilting object.  In \S\ref{subsection:PerpendicularTubes} we will consider the case where $\EE^\perp$ is a direct sum of tubes.  Then in \S\ref{subsection:OnlyOneExceptional}, we will consider the case where $\EE$ consists of a single indecomposable object $E$ and we will show that $E$ lies in a \emph{simple} tube of rank two in $\AA$.  We will use these results in \S\ref{subsection:LargerSequence} to show that there are ``enough'' simple tubes (as in Lemma \ref{lemma:SS}).  Afterward, the proof of the classification is similar to the proof of Theorem \ref{theorem:NoExceptionals}.

\subsection{A decomposition theorem}  In this section, we will deal with numerically finite hereditary categories with Serre duality and with nonzero projective objects.  Our main result is Theorem \ref{theorem:Decomposition} below, which says that such a category is a direct sum of a subcategory without nonzero projective objects and a subcategory with enough projective objects.  The theorem is thus an analogue of \cite[Theorem II.4.2]{ReVdB02} for non-noetherian categories.

\begin{theorem}\label{theorem:Decomposition}
Let $\AA$ be an abelian hereditary category.  Let $\PP$ be the full subcategory of $\AA$ consisting of projective objects and let $\BB$ be the the abelian subcategory of $\AA$ generated by $\BB$.  The subcategory $\BB$ is a Serre subcategory of $\AA$ and every object $B \in \BB$ has a projective resolution of length at most one.

If furthermore $\AA$ is Ext-finite, satisfies Serre duality, and $\PP$ has only finitely many isomorphism classes of indecomposable objects, then
\begin{enumerate}
\item there is a finite-dimensional hereditary algebra $\Lambda$ such that $\BB \cong \mod \Lambda$,
\item every injective object lies in $\BB$,
\item every object in $\BB$ has an injective resolution of length at most one,
\item $\BB^\perp = {}^\perp \BB$ and $\AA \cong \BB \oplus \BB^\perp \cong \BB \oplus {}^\perp \BB$.
\end{enumerate}
\end{theorem}

\begin{proof}
Let $\tilde{\PP}$ be the full subcategory of $\AA$ whose objects are quotient objects of projective objects.  Since $\AA$ is hereditary, the kernel of such an epimorphism $P_0 \twoheadrightarrow B$ is again projective and hence every object $B \in \tilde{\PP}$ fits into a short exact sequence
$$0 \to P_1 \to P_0 \to B \to 0$$
where $P_0,P_1 \in \PP$.  It is clear that $\tilde{\PP}$ is closed under extensions.  We will now check that $\tilde{\PP}$ is closed under subobjects.  Let $B'$ be a subobject of $B \in \tilde{\PP}$, and let $P_0 \to B$ be an epimorphism with $P_0 \in \PP$.  Let $P_0'$ be the pullback of the induced co-span, thus:
$$\xymatrix{{P'_0} \ar@{-->}[r] \ar@{-->}[d] & P_0 \ar[d] \\
{B'} \ar[r] & B.}$$
Since $B' \to B$ is a monomorphism, so is $P'_0 \to P_0$ and thus $P'_0$ is a projective object.  Since $P_0 \to B$ is an epimorphism, so is $P'_0 \to B'$ and thus $B' \in \tilde{\PP}$, as required.

We see that $\tilde{\PP}$ is an Serre subcategory of $\AA$ and thus $\tilde{\PP} = \BB$.

Assume now that $\AA$ is Ext-finite and that $\PP$ has only finitely many isomorphism classes of indecomposable objects.  Let $P$ be an additive generator for $\PP$, thus every object in $\PP$ is a direct summand of a direct sum of copies of $P$.  Since $P$ is a projective generator for $\BB$, we have $\BB \cong \mod \End(P)$.  In particular, $\End(P)$ is a hereditary algebra.  

Assume now that $\AA$ has Serre duality.  We will show that all injective objects of $\AA$ lie in $\BB$.  Let $\II$ be the category of injectives of $\AA$ and let $I$ be an additive generator for $\II$ (here we use Serre duality and the correspondence between projective objects and injective objects from Proposition \ref{proposition:SerreDuality} to show such an additive generator exists).  Consider the exact sequence
$$0 \to \ker f \to \Hom(P,I) \otimes P \stackrel{f}{\rightarrow} I \to \coker f \to 0$$
where $f$ is the evaluation morphism.  Since $\ker f$ is a subobject of a projective object, it itself is projective.  Likewise, $\coker f$ is injective.  Using the lifting property for projectives, we know that $\Hom(P, \coker f) = 0$.  It then follows from Proposition \ref{proposition:SerreDuality} that $\coker f = 0$.  We conclude that $I \in \tilde{\PP} = \BB$.

Let $\tilde{\II}$ be the full subcategory of $\AA$ whose objects are subobjects of objects in $\II$.  Dual to the case of $\tilde{\PP}$, we find that $\tilde{\II}$ is a Serre subcategory of $\AA$.  Since $I \in \tilde{\PP}$, we find
$$\tilde{\II} = \tilde{\PP} = \BB.$$
Moreover, $I$ is a tilting object for $\BB$.  Similar to the above considerations, for every object $B \in \BB$, there is a short exact sequence
$$0 \to B \to I_0 \to I_1 \to 0$$
where $I_0, I_1 \in \II$.

We will now show that $\BB^\perp = {}^\perp \BB$.  We start with $\BB^\perp \subseteq {}^\perp \BB$.  Let $C \in \BB^\perp$ and let $B \in \BB$ be any object.  For any nonzero morphism $g \in \Hom(C,B)$, the image $\im g$ lies in $\BB$.  If $\im g$ were nonzero, there would be a nonzero map $P \to \im g$.  Using the lifting property for projectives, we would find a nonzero map $P \to C$.  However, we have assumed that $C \in \BB^\perp$ and may thus conclude that $\im g = 0$.  For any nonzero extension $g \in \Ext^1(C,B)$, we have a nonzero map $C \to I_1$, where
$$0 \to B \to I_0 \to I_1 \to 0$$
is an injective resolution of $B$.  However, we have already excluded the existence of these maps.  We conclude that $\BB^\perp \subseteq {}^\perp \BB$.  The other inclusion is shown dually.
\end{proof}

\begin{remark}
\begin{enumerate}
\item The above theorem fails in general when $\PP$ has infinitely many isomorphism classes of indecomposable objects.  For example, the properties (1) through (4) fail for the categories given in \cite[Example 4.16]{BergVanRoosmalen11}.  In general, there will be no right adjoint to the embedding $\BB \to \AA$.
\item Without the finiteness conditions on $\PP$, one does not know whether $\BB$ has Serre duality even if $\AA$ satisfies Serre duality.
\end{enumerate}
\end{remark}

The next corollary explains why we are interested in categories without nonzero projective objects.

\begin{corollary}
Let $\AA$ be an indecomposable numerically finite abelian hereditary category with Serre duality.  If $\AA$ has a nonzero projective direct summand, then $\AA \cong \mod \Lambda$ for a finite-dimensional hereditary algebra $\Lambda$.
\end{corollary}

We end this subsection with two corollaries of Theorem \ref{theorem:Decomposition} we will use later.

\begin{proposition}\label{proposition:TiltingNoProjectiveDirectSummands}
Let $\AA$ be an abelian Ext-finite hereditary category with a tilting object.  If $\AA$ has no nonzero projective-injective objects, then $\AA$ has a tilting object without projective direct summands.
\end{proposition}

\begin{proof}
Since $\AA$ has a tilting object, we know that $\AA$ is numerically finite and that $\Db \AA$ has a Serre functor (see for example \cite[Theorem 6.1]{Lenzing07}).  Since $\AA$ is numerically finite and nonisomorphic projective objects are linearly independent in $\Num \AA$, we know $\AA$ has only finitely many isomorphism classes of indecomposable projective objects.  We can thus apply Theorem \ref{theorem:Decomposition}.  Let $\BB$ be the Serre subcategory of $\AA$ generated by the projective objects.

Let $P$ and $I$ be additive generators for the category of projectives and injectives in $\AA$, respectively.  Let $T$ be a tilting object in $\AA$.  Consider the evaluation morphism $\varphi: \Hom(P,T) \otimes P \to T$ and the short exact sequence
$$0 \to \im \varphi \to T \to T' \to 0.$$
Applying the functor $\Hom(P,-)$ shows that $T' \in \BB^\perp = {}^\perp \BB$.  Since $\im \varphi \in \BB$ and $\Ext^1(T', \im \varphi) = 0$, we have $T \cong T' \oplus \im \varphi$.  One can check that $T'$ is a tilting object for $\BB^\perp$ by applying $\Hom(-,C)$ to the above short exact sequence and using that $T$ is a tilting object.

Since $I$ is a tilting object of $\BB$ and $T'$ is a tilting object of $\BB^\perp$, we know that $T' \oplus I$ is a tilting object for $\AA$.  Since neither $I$ nor $T'$ have projective direct summands, we can conclude the same about $T' \oplus I$.
\end{proof}

\begin{corollary}\label{corollary:NoProjectiveSummands}
Let $\AA$ be an abelian indecomposable Ext-finite hereditary category with a tilting object.  If $\AA$ is not equivalent to $\mod \Gamma$ where $\Gamma$ is the algebra of upper-triangular $n \times n$-matrices over $k$ (for some $n \geq 1$), then $\AA$ has a tilting object without projective direct summands.
\end{corollary}

\begin{proof}
As in the proof of Proposition \ref{proposition:TiltingNoProjectiveDirectSummands}, we know that $\AA$ satisfies all conditions from Theorem \ref{theorem:Decomposition}.  Since the case where $\AA$ has no nonzero projectives is trivial, we may assume that $\AA \cong \mod \Lambda$ for a finite-dimensional hereditary algebra $\Lambda$ (by Theorem \ref{theorem:Decomposition}).  

By Proposition \ref{proposition:TiltingNoProjectiveDirectSummands}, it suffices to show that $\AA$ has no nonzero projective-injective objects.  According to \cite[Propositions III.1.1 and III.1.3]{AuslanderReiten73}, $\AA$ can only have projective-injective objects if $\AA \cong \mod \Gamma$ where $\Gamma$ is the ring of upper-triangular $n \times n$-matrices over $k$ (for some $n \geq 1$).
\end{proof}

\begin{remark}\label{remark:NoProjectiveSummands}
The algebra $\Gamma$ of upper-triangular $n \times n$-matrices over $k$ mentioned in Corollary \ref{corollary:NoProjectiveSummands} is the path algebra of an $A_n$-quiver with linear orientation.  We note that $\Gamma$ is fractionally Calabi-Yau of dimension $\frac{n-1}{n+1}$ (see for example \cite[Examples 8.3(2)]{Keller05}) in the sense that $\bS^{n+1} \cong [n-1]$ in $\Db \mod \Gamma$.
\end{remark}

\subsection{\texorpdfstring{$\EE^\perp$}{The perpendicular} is a direct sum of tubes}\label{subsection:PerpendicularTubes}  Let $\AA$ be a numerically finite abelian hereditary category with Serre duality, and let $\EE$ be a maximal exceptional sequence in $\AA$.  The first case we will consider is where $\EE^\perp$ is a direct sum of tubes.

We start with the following lemma.

\begin{lemma}\label{lemma:WhatIfOnlyTubes}
Let $\BB$ be the abelian (and hereditary) subcategory of $\AA$ generated by $\EE$.  The natural embedding $\BB \to \AA$ induces an isomorphism $\Num \BB \to \Num \AA$.  Furthermore, $\Db \BB$ has a Serre functor $\bS_\BB$ and for each element $B \in \Db \BB$, we have $[\bS_\BB B] = [\bS B]$ in $\Num \Db \AA$.
\end{lemma}

\begin{proof}
Let $\EE = (E_1, E_2, \ldots, E_n)$.  Since $\EE$ is a full exceptional sequence in $\BB$, we know that $\Num \BB \cong \oplus_{i=1}^n \bZ [E_i]$.  It follows from Proposition \ref{proposition:NumAndPerpendicular} that $\Num \AA / \Num \BB \cong \Num \EE^\perp$.  Since we have assumed that $\EE^\perp$ is direct sum of tubes, we know that $\Num \EE^\perp = 0$.  This implies that the embedding $\Num \BB \to \Num \AA$ is an isomorphism.

Since $\BB$ has a tilting object, $\Db \BB$ has a Serre functor (see for example \cite[Theorem 6.1]{Lenzing07}).  For an object $B \in \Db \BB$, we have $\chi(-,[\bS_\BB B]) = \chi([B],-) = \chi(-,[\bS B])$ as functions $\Num \Db \BB \to \bZ$.  Since $\chi(-,-)$ is nondegenerate, this shows that $[\bS_\BB B] = [\bS B]$ in $\Num \Db \BB$ and thus also in $\Num \Db \AA$.
\end{proof}

This implies the following proposition.

\begin{proposition}\label{proposition:WhatIfTubes}\label{Proposition:WhatIfOnlyTubes}
Assume that $\AA$ has no nonzero projective objects and assume that $\EE^\perp$ is a direct sum of tubes.  Then $\AA$ is a direct sum of categories, each of which is equivalent to either
\begin{enumerate}
\item a tube, or
\item a hereditary category with a tilting object.
\end{enumerate}
\end{proposition}

\begin{proof}
Let $\BB \subseteq \AA$ be the abelian subcategory generated by $\EE$, and let $\BB'$ be a connected component of $\BB$.  First assume that $\BB' \not\cong \mod \Gamma$ where $\Gamma$ is the algebra of upper-triangular $n \times n$-matrices over $k$.  We claim that $\BB'$ is a direct summand of $\AA$.

By Corollary \ref{corollary:NoProjectiveSummands}, we know that there is a tilting object $T$ in $\BB'$ such that no direct summand of $T$ is projective in $\BB'$.  We see that $T$ itself is a spanning class for $\BB'$ in the sense of \cite{Bridgeland99} (see Definition \ref{definition:SpanningClass}).

Let $i:\BB' \to \AA$ be the embedding.  Note that since $\BB'$ is a connected component of $\BB$, that the Serre functor on $\Db \BB$ restricts to a Serre functor on $\BB'$, or for any object $B' \in \BB'$ without projective direct summands, we have $\tau_\BB B' \cong \tau_{\BB'} B'$.  Lemma \ref{lemma:WhatIfOnlyTubes} implies that $[i \circ \bS_{\BB'} (T[0])] = [\bS_\AA \circ i (T[0])]$, or equivalently that $[i \circ \tau_{\BB'} (T)] \cong [\tau_\AA \circ i (T)]$ (we use here that $\AA$ has no projective objects and no direct summands of $T$ are projective in $\BB'$).  We can now apply Proposition \ref{Proposition:ExceptionalInK} to see that $i \circ \bS_{\BB'} (T[0]) \cong \bS_\AA \circ i (T[0])$.  It follows from Corollary \ref{corollary:SpanningClass} that $i:\BB' \to \AA$ induces an equivalence between $\BB'$ and a direct summand of $\AA$.

Assume next that $\BB'$ is not a direct summand of $\AA$ (thus in particular $\BB' \cong \mod \Gamma$ where $\Gamma$ is the algebra of upper-triangular $n \times n$-matrices over $k$).  In this case, the action of $\t$ on $\Num(\BB')$ is periodic (since $\Db \mod \Gamma$ is fractionally Calabi-Yau, see for example \ref{remark:NoProjectiveSummands}).  Since the action of $\t$ agrees on $\Kred(\BB')$ and $\Kred(\AA)$, and since $\AA$ has no projective objects, we know that every object in the essential image of $i$ has a finite $\t$ orbit.  It follows from Theorem \ref{theorem:TubeCriterium} that every object lies in a tube.  These categories are understood via \cite[Theorem 1.1 and Remark 6.19]{vanRoosmalen12}.
\end{proof}

\begin{corollary}\label{corollary:WhatIfTubes}
Let $\AA$ be indecomposable and assume that $\EE^\perp$ is a direct sum of tubes.  Then $\AA$ is equivalent to either
\begin{enumerate}
\item a tube, or
\item a hereditary category with a tilting object.
\end{enumerate}
\end{corollary}

\begin{proof}
By Theorem \ref{theorem:Decomposition}, we may assume that either $\AA$ has a tilting object, or that $\AA$ does not have any projective objects.  The category $\AA$ is thus equivalent to one of the categories in Proposition \ref{proposition:WhatIfTubes}.
\end{proof}

\begin{corollary}
Assume that $\AA$ is indecomposable.  If $\EE^\perp$ contains a connected component which is a homogeneous tube, then $\AA$ is a tube.
\end{corollary}

\begin{proof}
Directly from Corollary \ref{corollary:WhatIfTubes}.
\end{proof}

\subsection{\texorpdfstring{$\EE$}{E} consists of a single object}\label{subsection:OnlyOneExceptional} Assume now that $\AA$ is indecomposable, and that there is a maximal exceptional sequence $\EE$ consisting of a single object $E$.  The situation where $E^\perp$ contains only tubes is understood by Proposition \ref{Proposition:WhatIfOnlyTubes}, so we can assume that $E^\perp$ contains at least one component which is not a tube.  By Theorem \ref{theorem:NoExceptionals}, this component $\CC$ is derived equivalent to $\coh \bX$ for a smooth projective curve $\bX$.  We want to investigate the relation between the simple tubes in $\AA$ and the simple tubes in $E^\perp$ (see Proposition \ref{proposition:EnoughOneSphericals} below).

Since $E \not\in E^\perp$, we will instead consider $M \cong T^*_{\bS E} E$, thus $M$ is the image of $E$ under the right adjoint $T^*_{\bS E}:\Db \AA \to \Db E^\perp$ to the embedding $\Db E^\perp \to \Db \AA$.   We will use the description given in Lemma \ref{lemma:MiddleTermPerpendicular} below.  We will use $M$ as a means to convey information from $E^\perp$ (using the classification in Theorem \ref{theorem:NoExceptionals}) to $\AA$.  Perhaps a surprising result is that $E$ lies in a \emph{simple} tube of rank 2 in $\AA$ (see Proposition \ref{proposition:LiesInASimpleTube} below), and consequently, $M$ will lie in a homogeneous simple tube in $E^\perp$.

\begin{lemma}\label{lemma:MiddleTermPerpendicular}
Let $E \in \Ob \AA$ be an exceptional object and assume that $E^\perp$ has no exceptional objects.  If $E$ is not projective then the following statements hold:
\begin{enumerate}
\item $M$ is the middle term of the Auslander-Reiten sequence $0 \to \tau E \to M \to E \to 0$ in $\AA$,
\item $\dim \Hom(E, \t^2 E) = \dim \Ext^1(M,M) \not= 0$,
\item $\Hom(\t E, E) = 0$,
\item $\chi(M,M) \leq 0$, and
\item $M$ is indecomposable (thus $E$ is peripheral in $\AA$).
\end{enumerate}
\end{lemma}

\begin{proof}
Since $\Hom(E,\bS E) \cong \Hom(E,E)^*$, there is (up to isomorphism) a unique nonsplit triangle $\tau E \to M \to E \to \bS E$ in $\Db \AA$, which is an Auslander-Reiten triangle.  It is easily checked from this that $M \cong T_{\bS E}^* E$.  This triangle restricts to an Auslander-Reiten sequence $0 \to \tau E \to M \to E \to 0$, proving the first statement.
 
From this Auslander-Reiten sequence, we obtain:
\begin{align*}
\dim \Hom(M,\t M) &= \dim \Hom(\t E, \t M) + \dim \Hom(E, \t M) \\
&= \dim \Hom(E, M) + \dim \Ext^1(M, E).
\end{align*}

Since $E$ is exceptional, we know that $\Hom(E, \t E) = 0$ and $\dim \Hom(E,E) = 1$.  Applying the functor $\Hom(E, -)$ to the Auslander-Reiten sequence $0 \to \tau E \to M \to E \to 0$ thus gives $\dim \Hom(E, M) = 0$.

Finally, applying $\Hom(-,E)$ gives
$$\dim \Ext^1(M,E) = \dim \Ext^1(\t E, E) + \dim\Ext^1(E,E) = \dim \Hom(E, \t^2 E).$$
Combining these equalities yields $\dim \Hom(E, \t^2 E) = \dim \Ext^1(M,M)$, as requested.  Furthermore, we have assumed that $E^\perp$ has no exceptional objects, so that $M \in E^\perp$ implies that $\Ext^1(M,M) \not= 0$.

We will now prove that $\Hom(\t E, E) = 0$.  Seeking a contradiction, let $f: \t E \to E$ be a nonzero morphism.  Such a nonzero morphism $\t E \to E$ cannot be an epimorphism, since then there would be a nonzero composition $\t E \twoheadrightarrow E \to \t^2 E$, contradicting $\Hom(\t E, \t^2 E) \cong \Ext^1(E,E)^* = 0$.  Analogously, a monomorphism $\t E \hookrightarrow E$ would lead to a nonzero composition $\t E \hookrightarrow E \to \t^2 E$, again contradicting $\Ext^1(E,E) \not= 0$.

For a nonzero morphism $f: \t E \to E$, one checks easily that $\im f \in E^\perp$ so that $\Ext^1(\im f, \im f) \not= 0$ since $E^\perp$ has no exceptional objects.  There is thus a nonzero morphism $\im f \to \t \im f$ which we can use to find a nonzero morphism $\t E \twoheadrightarrow \im f \to \t \im f \hookrightarrow \t E$.  Since $\t E$ is exceptional, this composition is an isomorphism and hence $\t \im f \to \t E$ is a (split) epimorphism and thus an isomorphism.  Hence, $f$ is a epimorphism, but we have already established that $\Hom(\t E, E)$ has no epimorphisms.  We conclude that $\Hom(\t E,E) = 0$.

Applying $\Hom(\t E,-)$ to the Auslander-Reiten sequence $0 \to \t E \to M \to E \to 0$ shows that $\dim \Hom(\t E, M) = 1$ and hence $M$ is indecomposable.
\end{proof}

We will now look closer at the Cartan matrix of $\AA$.

\begin{lemma}\label{lemma:CartanReduced}
One can extend $[E] \in \Num \AA$ to a basis of $\Num \AA$ such that the Cartan matrix is block diagonal.  Moreover, the block in the upper left corner is either $(1)$, or
$$\begin{pmatrix}
1 & 0 & 0 \\
a & 0 & -1 \\
b & 1 & 1-g
\end{pmatrix}$$
for integers $a,b,g$, and where $g \geq 1$ where the first basis element of $\Num \AA$ is given by $[E]$.
\end{lemma}

\begin{proof}
First, note that $E^\perp$ satisfies the conditions of Theorem \ref{theorem:NoExceptionals} so that every connected component $\BB$ of $E^\perp$ is derived equivalent to either a homogeneous tube or the category of coherent sheaves on a smooth projective variety $\bX$ of genus $g \geq 1$.  The numerical Grothendieck group $\Num \BB$ is zero in the former case, and isomorphic to $\bZ^2$ in the latter case.

We consider the case where $\Db \BB \cong \Db \coh \bX$.  We can choose a basis in $\Num \BB$ such that the corresponding Cartan matrix is given by
$$\begin{pmatrix}
0 & -1 \\
1 & 1-g
\end{pmatrix}.$$
Indeed, it follows from the Riemann-Roch theorem that the above matrix is the Cartan matrix with respect to the basis $[k(P)]$ and $[\OO_\bX]$, where $k(P)$ is the simple sheaf supported on a closed point $P \in \bX$ and $\OO_\bX$ is the structure sheaf (see Example \ref{example:CartanCoxeterCurve}).  We see that this gives a basis of $\Num E^\perp$ such that the Cartan matrix is block diagonal, where all blocks are given by matrices of the form described above.

Following Proposition \ref{proposition:NumAndPerpendicular}, we know that $\Num \AA \cong \bZ[E] \oplus \Num E^\perp$.  This extends the basis of $\Num E^\perp$ to a basis for $\Num \AA$.  We will assume that $[E]$ is the first of these basis elements, thus the upper left entry in the Cartan matrix of $\AA$ is $\chi(E,E)$.

If $\BB'$ is a connected component of $E^\perp$ not containing $M$, then we have
$$\Hom(\BB',M) = \Hom(M, \BB') =  \Ext^1(\BB',M) = \Ext^1(M,\BB') = 0.$$
Using the Auslander-Reiten sequence from Lemma \ref{lemma:MiddleTermPerpendicular}, we find that the above equalities hold after replacing $M$ by $E$.  Thus, for a connected component $\BB'$ of $E^\perp$ not containing $M$, we find $\chi(\BB',E) = 0$ and $\chi(E,\BB') = 0$.

We have the following possibilities.  If the connected component $\BB$ is a homogeneous tube, then $\chi(-,[M])=0$ on $\Num \BB$.  Using the Auslander-Reiten sequence from Lemma \ref{lemma:MiddleTermPerpendicular} and $\chi(-,[\t E]) = -\chi([E],-)=0$, we find $\chi(-,[E])=0$ in $\Num \BB$.  In this case, the Cartan matrix of $\AA$ is given by the Cartan matrix of $E^\perp$ with an additional 1 on the diagonal.  If $M$ lies in a component of the form $\coh \bX$ for a smooth projective curve $\bX$, then we obtain the Cartan matrix of $\AA$ from the Cartan matrix of $E^\perp$ by replacing the corresponding block by
$$\begin{pmatrix}
1 & 0 & 0 \\
a & 0 & -1 \\
b & 1 & 1-g
\end{pmatrix}$$
\end{proof}

We can use this form of the Cartan matrix to prove the following proposition.

\begin{proposition}\label{proposition:Eis2periodic}
We have $E \cong \t^2 E$.
\end{proposition}

\begin{proof}
We will first show that $\Hom(\t^2 E, E) \not= 0$ and that $\Hom(E,\t^2 E) \not= 0$.  The latter has been shown in Lemma \ref{lemma:MiddleTermPerpendicular}.

To show that $\Hom(E, \t^{-2} E) \not= 0$, we will look at the Cartan and the Coxeter matrix of $\AA$.  It follows from Lemma \ref{lemma:CartanReduced} that the Cartan matrix $A$ is block diagonal and it follows from Proposition \ref{proposition:CartanCoxeter} that the Coxeter matrix $C = -A^{-1} A^T$ follows this decomposition.  Since we are interested in $E$ and $\t^2 E$, we need only to consider the block of $A$ on the row and column corresponding to $[E]$.

If this block is $(1)$, then the corresponding block of the Coxeter matrix is $(-1)$ and we find $[\t^2 E] = [E]$, so that $\chi(\t^2 E,E) = 1$ and thus $\Hom(E, \t^{-2} E) \not= 0$.  Note that in this case, we can obtain $E \cong \t^2 E$ from Proposition \ref{Proposition:ExceptionalInK}.

The other case is where the block is given by
$$\begin{pmatrix}
1 & 0 & 0 \\
a & 0 & -1 \\
b & 1 & 1-g
\end{pmatrix}$$
The corresponding block of the Coxeter matrix $C = -A^{-1} A^T$ is
$$\begin{pmatrix}
-1 & -a & -b \\
a(1-g)+b & a^2(1-g)+ab+1 & (ab-2)(1-g)+b^2 \\
-a & -a^2 & 1-ab
\end{pmatrix}.$$
For the upper left block of the matrix $C^2$, we find:
$$\begin{pmatrix}
1-a^2(1-g) & * & * \\
a^3(1-g)^2+a(ab+2)(1-g) & * & * \\
-a^3(1-g) & * & *
\end{pmatrix}$$
Since $A$ and $C$ are given with respect to a basis of $\Num \AA$ of which $[E]$ is the first element, we know that $\chi(\t^2 E,E)$ is the upper left entry of the matrix $(C^2)^{T}A$.  We find $\chi(\t^2 E,E) = a^4(1-g)^2 + a^2(1-g)+1 > 0$.  Hence, $\Hom(E, \t^{-2} E) \not= 0$.

We now know that $\Hom(\t^2 E, E) \not= 0$ and $\Hom(E,\t^2 E) \not= 0$ and will use this to prove that $E \cong \t^2 E$.  Recall that we have shown in Lemma \ref{lemma:MiddleTermPerpendicular} that $\Hom(E, \t^{-1} E) = 0$, and thus also that $\Ext^1(E, \t^2 E) = 0$.  It follows from \cite[Lemma 4.1]{HappelRingel82} (see also \cite[Proposition 5.1]{Lenzing07}) that every nonzero morphism $E \to \t^{2} E$ is either an epimorphism or a monomorphism.

Let $f: E \to \t^{2} E$ be an epimorphism so that there is a nonzero composition $E \twoheadrightarrow \t^{2} E \to E$.  Since $E$ is exceptional, and thus every nonzero morphism is invertible, this implies that $E \cong \t^2 E$.  Likewise, if $f: E \to \t^{2} E$ is a monomorphism, then the composition $E \to \t^{-2} E \hookrightarrow E$ is nonzero and we again find that $E \cong \t^2 E$.
\end{proof}

\begin{corollary}\label{corollary:TubeRank2}
The object $E \in \AA$ lies in a tube of rank 2.
\end{corollary}

\begin{proof}
This follows from Proposition \ref{proposition:Eis2periodic} and Theorem \ref{theorem:TubeCriterium}.\end{proof}

\begin{corollary}\label{corollary:M1Spherical}
The object $M$ is 1-spherical in $E^\perp$.
\end{corollary}

\begin{proof}
We know from Proposition \ref{lemma:MiddleTermPerpendicular} that $ \dim \Hom(E, \t^2 E) = \dim \Ext^1(M,M)$ and from Proposition \ref{proposition:Eis2periodic} that $\dim \Hom(E, \t^2 E) = 1$.  Furthermore, $M$ lies in a category derived equivalent to one of the categories described in Theorem \ref{theorem:NoExceptionals}.  We conclude that $M$ is 1-spherical.
\end{proof}

\begin{proposition}\label{proposition:LiesInASimpleTube}
The object $E \in \AA$ lies in a simple tube of rank 2.
\end{proposition}

\begin{proof}
We know from Corollary \ref{corollary:TubeRank2} that $E$ lies in a tube $\TT$ of rank 2.  We will show that $\TT$ is a simple tube in $\AA$. 

Assume first that $M$ lies in a connected component of $E^\perp$ which is a (homogeneous) tube.  We will show that the category $\AA$ is decomposable and one of the components is given by the tube $\TT$.  Seeking a contradiction, let $A \in \AA$ be any indecomposable object not lying in the tube $\TT$, such that either $\Hom(A,M) \not= 0$ or $\Hom(M,A) \not= 0$.  We will begin by considering the former.

Since the tube is a connected component in $E^\perp$, we know that $A \not\in E^\perp$.  We apply the twist functor $T_E: \Db \AA \to \Db \AA$ to find a nonzero object $T_E A \in E^\perp$, fitting in a triangle
$$\RHom(E,A) \otimes_{k} E \to A \to T_E A \to \RHom(E,A) \otimes_{k} E[1].$$
Since $M \in E^\perp$, we see that $\Hom(T_E A,M) \not= 0$.  This implies that $T_E A$ lies in the tube $\TT$, and hence so does $A$.

We now turn our attention to the latter case, namely the case where $\Hom(A,M) \not= 0$.  We then know that $\tau A \not\in E^\perp = {}^\perp (\tau E)$.  We apply the twist functor $T^*_{\t E}: \Db \AA \to \Db \AA$ to find a nonzero object $T^*_{\t E} A \in {}^\perp (\tau E)$, fitting in a triangle
$$T^*_{\t E} A \to A \to \RHom(A,E)^* \otimes_{k} E \to T^*_{\t E} A[1].$$
Since $M \in E^\perp = {}^\perp (\tau E)$, we see that $\Hom(M,A) \cong \Hom(M, T^*_{\t E} A)$ and hence nonzero.  This implies that $T^*_{\t E} A$ lies in the same tube as $M$ in $E^\perp$ and according to the above triangle, also in $\AA$.  Thus, both $T^*_{\t E} A$ and $A$ lie in $\TT$.

In both cases, it follows from Proposition \ref{proposition:Summary} that $E$ lies in a simple tube in $\AA$.

Assume thus that $M$ lies in a connected component $\BB$ which is derived equivalent to $\coh \bX$ for a smooth projective curve.  It follows from Lemma \ref{lemma:CartanReduced} that the Cartan matrix of $\AA$ is a block diagonal matrix such that with respect to a suitably chosen basis of $\Num \AA$ (the first element of this basis is $[E]$), the upper left block is given by
$$A=\begin{pmatrix}
1 & 0 & 0 \\
a & 0 & -1 \\
b & 1 & 1-g
\end{pmatrix},$$
and the corresponding part of the Coxeter matrix $C = -A^{-1}A^T$ of $\AA$ is
$$C=\begin{pmatrix}
-1 & -a & -b \\
a(1-g)+b & a^2(1-g)+ab+1 & (ab-2)(1-g)+b^2 \\
-a & -a^2 & 1-ab
\end{pmatrix}.$$
We still have some liberty in choosing this basis.  We know from Corollary \ref{corollary:M1Spherical} that $M$ is 1-spherical in $E^\perp$.  We can thus choose $M$ as the object $X$ in the construction of the category $\HH$ so that $M$ is simple in $\HH$ (see Proposition \ref{proposition:SincereWhenNoExceptionals}).  Using the equivalence $\HH \cong \coh \bX$ from the proof of Theorem \ref{theorem:NoExceptionals} (see Remark \ref{remark:WhatIsHH}), we see that the basis $[M], [\OO_\bX]$ of $\Num (\coh \bX) \cong \Num \HH \cong \Num \BB$ can be chosen to obtain the above matrices.  With respect to this basis, we find 
$$[\t E] = -1 [E] + (a(1-g)+b)[M] - a[\OO_\bY].$$
Using that $[M] = [E] + [\t E]$ then yields $a=0$ and $b=1$.

We are now ready to show that $E$ lies in a simple tube.  We will use Proposition \ref{proposition:Summary}.  Seeking a contradiction, let $X,Y \in \AA$ be indecomposable objects, not lying in the tube containing $E$, such that $\Hom(X,E \oplus \t E) \not= 0$ and $\Hom(E \oplus \t E, Y) \not= 0$.  We write
\begin{align*}
[X] &= x_1 [E] + x_2 [M] + x_3 [\OO_\bY] \\
[Y] &= y_1 [E] + y_2 [M] + y_3 [\OO_\bY]
\end{align*}
Since $\Hom(X,E \oplus \t E) \not= 0$, Corollary \ref{corollary:TubesDirecting} shows that $\Ext^1(X,E \oplus \t E)= 0$ so that $\chi(X,M) > 0$.  Using the Cartan matrix, we find that $\chi(X,M) = x_3$, and thus $x_3 > 0$.  Similarly, one shows that $y_3 < 0$.

We will be looking at $\t^{2n} Y$ for $n \gg 0$.  The square of the Coxeter matrix is
$$C^2 = \begin{pmatrix}
1 & 0 & 0 \\
0 & 1 & 4g-3 \\
0 & 0 & 1
\end{pmatrix}$$
so that $[\t^{2n} Y] = [Y] + nx_3(4g-3)[M]$ and $\chi(X,\t^{2n} Y) = \chi(X,Y) + n(4g-3)x_3y_3$.  Thus, for some $n \gg 0$, we know that $\chi(X, \t^{2n} Y) < 0$ and hence $\Ext^1(X, \t^{2n} Y) \not= 0$.  Using Lemma \ref{lemma:Ringel}, we know that there is a path from $\t^{2n} Y$ to $X$.  Since $\t^2 E \cong E$, we find $\Hom(E \oplus \t E, \t^{2n} Y) \not= 0$.  This gives a path $\t^{2n} Y \to X \to E \oplus \t E \to \t^{2n} Y$, which contradicts Theorem \ref{theorem:DirectingTubes}.  Proposition \ref{proposition:Summary} now yields that $E \in \AA$ lies in a simple tube.
\end{proof}

\begin{corollary}\label{corollary:MIsSimple}
The object $M$ is simple in $E^\perp$.
\end{corollary}

The following proposition is one of the main reasons why we pay special attention to simple tubes.  Note that when we do not require $S$ to be simple in $E^\perp$, we cannot conclude that $S$ is 1-spherical in $\AA$ (see Example \ref{example:NotSimple} below).

\begin{proposition}\label{proposition:EnoughOneSphericals}
Let $S$ be a simple 1-spherical object in $E^\perp$.  If $S \not\cong M$, then $S$ is a simple 1-spherical object in $\AA$.
\end{proposition}

\begin{proof}
We know by Proposition \ref{proposition:LiesInASimpleTube} that $E$ lies in a simple tube in $\AA$, thus $M$ lies in a simple tube in $E^\perp$.  If $S$ is a simple 1-spherical object in $E^\perp$ and $S \not\cong M$, then Proposition \ref{proposition:SimpleTubesPerpendicular} implies that $S \in M^\perp \cap E^\perp$.  Let $B = E \oplus \t E$.  By Proposition \ref{proposition:LiesInASimpleTube}, we know that $\t B \cong B$.  Since $M^\perp \cap E^\perp = B^\perp$, we know that $S \cong \t S$ in $\AA$.  Since $S$ is simple, we know that $\Hom(S,S) \cong k$ and hence $S$ is 1-spherical in $\AA$.

It now follows from Proposition \ref{proposition:PerpendicularOfSimples} that $S$ is simple in $\AA$ since $S$ is simple in $B^\perp$.
\end{proof}

\begin{example}\label{example:NotSimple}
Let $\AA$ be a tube of rank 2, and let $E$ be any of the peripheral objects.  The category $E^\perp$ is a simple (and homogeneous) tube, but the simple object $S$ in $E^\perp$ is not simple in $\AA$.  Indeed, there is a short exact sequence $0 \to \t E \to S \to E \to 0$.
\end{example}

\subsection{\texorpdfstring{$\EE$}{E} does not consist of a single object}\label{subsection:LargerSequence} We will now consider the case where $\AA$ is indecomposable and contains a maximal strong exceptional sequence $\EE = (E_1,E_2, \ldots, E_n)$ which does not necessarily consist of a single object.  To avoid trivial cases, we will assume that $\EE \not= \emptyset$ (the sequence is not empty) and that $\EE^\perp \not= 0$ (thus $\AA$ does not have a tilting object).

Our goal in this subsection to prove Proposition \ref{proposition:EnoughSimpleTubes}, namely that each $E_i$ lies in a simple tube in $\AA$; this is a more general version of Proposition \ref{proposition:LiesInASimpleTube} above.

In what follows, let $\BB_i = {}^\perp (E_1,E_2, \ldots, E_{i-1}) \cap (E_{i+1}, E_{i+2}, \dots, E_n)^{\perp} \subseteq \AA$.  Note that $E_i \in \BB_i$, and that $\BB_i \cap E_i^\perp$ has no exceptional objects.  Thus the category $\BB_i$ is of the form as discussed in \S\ref{subsection:OnlyOneExceptional}.

However, recall that we assume that $\AA$ does not have any nonzero projective objects, but we do not know whether $\BB_i$ has any nonzero projective objects.

\begin{proposition}\label{proposition:EveryComponentHasSimpleTube}
Let $\AA$ be an indecomposable numerically finite abelian hereditary category with Serre duality.  Let $\EE$ be a maximal exceptional sequence in $\AA$ and assume that $\EE \not= \emptyset$ and $\EE^\perp \not= 0$.  Every connected component in $\EE^\perp \subseteq \AA$ has a simple tube (the tube is simple in $\EE^\perp$).
\end{proposition}

\begin{proof}
Let $\CC$ be a connected component of $\EE^\perp \subseteq \AA$.  Recall that $\AA$ has no propjective objects so that Proposition \ref{proposition:SerreDuality} implies that $\t: \Db \AA \to \Db \AA$ restricts to an autoequivalence $\tau: \AA \to \AA$ (Proposition \ref{proposition:SerreDuality}).  Since $\CC$ is a connected component of $\EE^\perp$, the embedding $\CC \to \EE^\perp$ has a left and a right adjoint.  Since the embedding $\EE^\perp \to \AA$ has a left and a right adjoint, we may infer the same about the embedding $\CC \to \AA$.

If $\CC \subseteq {}^\perp \EE$, then $\t \CC \in \EE^\perp$ so that $\CC$ is a connected component of $\AA$, and thus $\CC = \AA$ since $\AA$ is indecomposable.  In particular, $\EE = \emptyset$.  Since we have assumed this is not the case, we know that $\CC \not\subseteq {}^\perp \EE$, hence there is an $E_i \in \EE$ such that $\Hom(\CC,E_i) \not= 0$ or $\Ext^1(\CC,E_i) \not= 0$.  Let $i$ be the smallest such number.

We will work in the category $\BB_i = {}^\perp (E_1,E_2, \ldots, E_{i-1}) \cap (E_{i+1}, E_{i+2}, \dots, E_n)^{\perp}$.  We know that $E_i \in \BB_i$ and $\CC \subseteq \BB_i$ (by the minimality of $i$).  By Theorem \ref{theorem:Decomposition} we know that if $E_i$ is projective in $\BB_i$ that the additive category generated by $E_i$ is a connected component of $\BB_i$.  However, since we have $\Hom(\CC,E_i) \not= 0$ or $\Ext^1(\CC,E_i) \not= 0$, this is not the case.  Hence, $E_i$ not projective in $\BB_i$.

As before, let $M_i$ be the image of $E_i$ under the right adjoint to the embedding $\BB_i \cap E_i^\perp \to \BB_i$, thus $M_i$ is the middle term of an almost split sequence $0 \to \tau_i E_i \to M_i \to E_i \to 0$ where $\tau_i$ is the Auslander-Reiten translate in $\BB_i$.  Lemma \ref{lemma:MiddleTermPerpendicular} then shows that the object $M_i$ is indecomposable and has to lie in $\CC$.  Corollary \ref{corollary:MIsSimple} yields that $\CC$ has a simple tube.
\end{proof}

\begin{lemma}\label{lemma:GoodSimpleTubes}
Let $\AA$ be an indecomposable numerically finite abelian hereditary category with Serre duality.  Let $\EE$ be a maximal exceptional sequence in $\AA$ and assume that $\EE \not= \emptyset$ and $\EE^\perp \not= 0$.  There is a set $\SS' \in \AA$ of simple 1-spherical objects in $\AA$ such that $\SS' \subseteq \EE^\perp$ and such that ${}^\perp \SS' \cap \EE^\perp$ consists only of tubes.
\end{lemma}

\begin{proof}
Let $\SS \subseteq \ind \EE^\perp$ be a set of representatives of isomorphism classes of simple 1-spherical objects in $\EE^\perp$.  Proposition \ref{proposition:EveryComponentHasSimpleTube} shows that each connected component of $\EE^\perp$ has a simple tube, and Corollary \ref{corollary:OnesimpleMeansEnoughSimples} shows that ${}^\perp \SS \cap \EE^\perp = 0$.

Let $\SS' = \SS \cap {}^\perp \EE = \SS \cap \EE^\perp$.  We claim that $\SS \setminus \SS'$ has only finitely many indecomposables.  Indeed, let $S \in \SS \setminus \SS'$.  This means that there is an $E_i \in \EE$ such that $\Hom(S,E_i) \not= 0$ or $\Ext^1(S,E_i) \not= 0$.  We will choose a minimal $i$ with this property and work in the category $\BB_i = {}^\perp (E_1,E_2, \ldots, E_{i-1}) \cap (E_{i+1}, E_{i+2}, \dots, E_n)^{\perp}$.

We claim that $E_i$ is not projective in $\BB_i$.  Indeed, assume that $E_i$ is projective in $\BB_i$.  By Theorem \ref{theorem:Decomposition}, this implies that the additive category generated by $E_i$ is a connected component of $\BB_i$.  This contradicts that $\Hom(S,E_i) \not= 0$ or $\Ext^1(S,E_i) \not= 0$.  We may conclude that $E_i$ is not projective in $\BB_i$.

As before, let $M_i$ be the middle term in the Auslander-Reiten sequence
$$0 \to \t_i E_i \to M_i \to E_i \to 0$$
in $\BB_i$, where $\t_i$ is the Auslander-Reiten translate in $\BB_i$.  As in Lemma \ref{lemma:MiddleTermPerpendicular}, we know that $M_i$ is indecomposable, and Corollary \ref{corollary:MIsSimple} shows that $M_i \in \BB_i \cap E_i^\perp$ lies in a simple tube.  It follows from $\Hom(S,E_i) \not= 0$ or $\Ext^1(S,E_i) \not= 0$ that $\Hom(S,M_i) \not= 0$ or $\Ext^1(S,M_i) \not= 0$, thus $M_i \not\in S^\perp$.

It follows from Proposition \ref{proposition:EnoughOneSphericals} that either $S$ is a simple 1-spherical object in $\BB_i \cap E_i^\perp$, or that $S \cong M_i$.  Since we know that $M_i \not\in S^\perp$, Proposition \ref{proposition:SimpleTubesPerpendicular} implies that $S \cong M_i$.

We conclude that at most $n$ objects will be removed when going from $\SS$ to $\SS' = \SS \cap \EE^\perp$.  By Theorem \ref{theorem:NoExceptionals}, we know that the only connected components of $\EE^\perp$ are derived equivalent to tubes or to $\coh \bX$ for a smooth projective curve $\bX$.  By proposition \ref{proposition:EveryComponentHasSimpleTube}, each of these components has a simple 1-spherical object in $\EE^\perp$.  The statement then follows from Corollary \ref{corollary:PerpendicularToSimpleTubeNoExceptionals} (here we use that $\Db \coh \bX$ has infinitely many nonisomorphic 1-spherical objects).
\end{proof}

\begin{proposition}\label{proposition:EnoughSimpleTubes}
Let $\AA$ be indecomposable. If $\AA$ does not have a tilting object, then $\AA$ has a set $\SS$ of pairwise perpendicular simple tubes such that $\SS^\perp = 0$.  Moreover, each object $E_i$ in $\EE$ lies in a tube of $\SS$.
\end{proposition}

\begin{proof}
Let $\SS'$ be as in the statement of Lemma \ref{lemma:GoodSimpleTubes}.  We can apply Proposition \ref{Proposition:WhatIfOnlyTubes} to the category ${}^\perp \SS'$ to see that ${}^\perp \SS'$ is equivalent to a direct sum of tubes and hereditary categories with a tilting object.  However, the embedding ${}^\perp \SS' \to \AA$ commutes with the Serre functor (Proposition \ref{proposition:PerpendicularOnSpherical}) and, since a hereditary category with a tilting object is saturated, each such component will be a direct summand of $\AA$.  Since we assume that $\AA$ is indecomposable and does not have a tilting object, we know that ${}^\perp \SS'$ is a direct sum of tubes.  We know that each $E_i$ lies in ${}^\perp \SS'$, and hence each object in $\EE$ lies in a tube of ${}^\perp \SS'$.  Furthermore, $\SS'$ consists of 1-spherical objects and since the embedding ${}^\perp \SS' \to \AA$ commutes with the Serre functor, we may conclude that each $E_i$ lies in a tube of $\AA$.

Since each of the tubes in ${}^\perp \SS'$ is simple in ${}^\perp \SS'$, Proposition \ref{proposition:PerpendicularOfSimples}  implies that these are simple in $\AA$.  It is now easy to see that the tubes in $\SS'$ together with the tubes in ${}^\perp \SS'$ form the set of orthogonal simple tubes from the statement of the proposition.
\end{proof}

\subsection{A torsion theory when \texorpdfstring{$\AA$}{A} does not have a tilting object}   In this subsection, we will consider the case where $\AA$ does not have a tilting object.  This setting resembles that of \S\ref{section:NoExceptionals}, and we will follow the steps of the proof of Theorem \ref{theorem:NoExceptionals} closely.  Since the case where $\AA$ is a single tube is easily dealt with, we will exclude this case.

Proposition \ref{proposition:EnoughSimpleTubes} implies that $\AA$ has at least one simple tube.  Let $S$ be the generalized 1-spherical object of that tube (see Proposition \ref{proposition:SphericalsInTubes}).

Using these objects, we can define a $t$-structure with hereditary heart $\HH = \HH(S)$ as in Construction \ref{construction:Heart} (see Proposition \ref{proposition:SemiSimpleInHeart}).

The following proposition is analogous to Proposition \ref{proposition:SincereWhenNoExceptionals}.

\begin{proposition}\label{proposition:SincereWhenExceptionals}
Let $\AA$ be an abelian hereditary category with Serre duality and without a tilting object.  There is a $t$-structure with hereditary heart $\HH$, derived equivalent to $\AA$, and a set $\SS$ of minimal 1-spherical objects such that
\begin{enumerate}
\item each $S \in \SS$ is semi-simple in $\HH$,
\item $\Hom(S_1, S_2) = 0 = \Ext^1(S_1,S_2)$, for every two nonisomorphic $S_1,S_2 \in \SS$,
\item for all $A \in \HH$, there is an $S \in \SS$ such that $\Hom(A, S) \not= 0$, and
\item all exceptional objects of $\HH$ are contained in the abelian subcategory generated by $\SS$.
\end{enumerate}
\end{proposition}

\begin{proof}
It follows from Proposition \ref{proposition:SemiSimpleInHeart} that there is at least one semi-simple minimal 1-spherical $S$ in $\HH$.  All the simple tubes in $\HH$ are perpendicular to $S$ (see Proposition \ref{proposition:SimpleTubesPerpendicular}).  It follows from Proposition \ref{proposition:EnoughSimpleTubes} that ${}^\perp \SS = \SS^\perp = 0$.  In order to show that $\SS$ satisfies all conditions in the statement of the proposition, we need to show that every exceptional object in $\HH$ lies in a tube generated by $\SS$.

Let $E_1' \in \HH$ be any exceptional object.  By Proposition \ref{proposition:SequenceMeansObject}, there is a partial tilting object $E'$ in $\HH$ which contains $E_1'$ as a direct summand and such that $(E')^\perp$ does not have any exceptional objects.  It follows from Proposition \ref{proposition:EnoughSimpleTubes} that $E'_1$ lies in a tube generated $\SS$.  This finishes the proof.
\end{proof}

We can define a torsion theory on $\HH$ as in \S\ref{subsection:Tilt}:
\begin{eqnarray*}
\TT &=& \{A \in \Ob \HH \mid \mbox{$A$ has finite length}\} \\
\FF &=& \{A \in \Ob \HH \mid \mbox{every nonzero direct summand of $A$ has infinite length}\}.
\end{eqnarray*}
Note that $\tau \TT = \TT$ and thus also $\tau \FF = \FF$.  It then follows from Serre duality that $\Ext^1(\FF,\TT) = 0$ such that the objects of $\TT$ act as injectives with respect to $\FF$.  In particular, every object in $\HH$ is a direct sum of a torsion and a torsion-free object.

We will now consider the quotient category $\HH / \TT$.  As in Lemma \ref{lemma:QuotientSemiSimple}, we find that the quotient $\HH / \TT$ is semi-simple.

For an object $F \in \HH/ \TT$,  let $\tilde{F}$ be a lift of $F$ in $\HH$. Since every object in $\HH$ is a direct sum of an object in $\TT$ and an object in $\FF$, we can choose the lift $\tilde{F}$ to be in $\FF$.  We put
$$w(F) = (\dim \Hom(\tilde{F}, S))_{S \in \SS}.$$
where $\SS \subseteq \HH$ is a set of representatives of isomorphism classes of all (semi-simple) minimal 1-spherical objects in $\HH$.

As before, we have the following lemma (see Lemma \ref{lemma:QuotientHasSimple}).

\begin{lemma}\label{lemma:QuotientHasSimple2}
Assume that there is an object $E \in \HH/ \TT$ such that $w(E)$ is bounded.  Then $\HH/ \TT$ contains a simple object.
\end{lemma}

\begin{proposition}\label{proposition:BoundedWhenExceptionals}
For every object $H \in \HH$, $w(H) = (\dim \Hom(H, S))_{S \in \SS}$ is bounded.
\end{proposition}

\begin{proof}
Since the elements of $\SS$ are mutually perpendicular and $\HH$ is numerically finite, we know that only finitely many objects in $\SS$ are not 1-spherical objects.  Indeed, each element of $\SS$ which is not 1-sperical (but rather generalised 1-spherical) lies in a tube with exceptional objects.  Since the different tubes are perpendicular, these exceptional objects correspond to linearly independant elements in $\Num \HH$ (see \ref{proposition:NumAndPerpendicular}).  Let $\SS'$ be the subset $\SS$ consisting of all 1-spherical objects, thus $\SS'$ is a cofinite subset of $\SS$.

Let $\BB$ be the full abelian subcategory of $\AA$ generated by $\EE$.  We know by Proposition \ref{proposition:SincereWhenExceptionals} that all exceptional objects lie in $\TT$.  Since then $\BB$ is contained in the simple tubes of $\AA$, we know that $\BB$ is a Serre subcategory of $\AA$, and since $\BB$ is generated by an exceptional sequence, the embedding $\BB \to \AA$ has a left adjoint $L: \AA \to \BB$.  We infer that ${}^\perp \BB \to \AA$ admits a right adjoint $R: \AA \to \BB^\perp$ such that for each $H \in \HH$ there is a short exact sequence
$$0 \to R(H) \to H \to L(H) \to 0.$$
We know that ${}^\perp \BB$ is derived equivalent to a direct sum of categories listed in Theorem \ref{theorem:NoExceptionals}, and hence we know that $(\dim \Hom(R(H), S'))_{S' \in \SS'}$ is bounded.  Since $L(H) \in \TT$, we also know that $(\dim \Hom(L(H), S'))_{S' \in \SS'}$ is bounded.  We then know that $(\dim \Hom(H, S'))_{S' \in \SS'}$ is bounded.  Since $\SS \setminus \SS'$ is finite, we may conclude that $(\dim \Hom(H, S))_{S \in \SS}$ is bounded.
\end{proof}

\subsection{Classification} Let $\AA$ be a connected hereditary category with Serre duality, linear over an algebraically closed field $k$.  Let $\EE = (E_i)_{i=1 \ldots n}$ be a (finite) maximal exceptional sequence.  It follows from Proposition \ref{proposition:SequenceMeansObject} that we may assume that $\EE \subseteq \AA$ and that $\EE$ is a strong exceptional sequence.  In particular, $E = \oplus_{i=1}^n E_i$ is a partial tilting object in $\AA$.

We can now prove our main theorem.

\TheoremNumFinite*

\begin{proof} Without loss of generality, assume that $\AA$ is not derived equivalent to a tube.  If $\AA$ has a tilting object, then the classification follows from \cite{Happel01} (see Theorem \ref{theorem:Happel}).  Thus, assume that $\AA$ does not have a tilting object.  We can repeat the proof of Proposition \ref{proposition:NoExceptionals} to show that $\AA$ is noetherian.  Here, we combine Proposition \ref{proposition:BoundedWhenExceptionals} and Lemma \ref{lemma:QuotientHasSimple2} to find a simple object in $\HH / \TT$.  The classification then follows from \cite{ReVdB02}.
\end{proof}

\def\cprime{$'$}
\providecommand{\bysame}{\leavevmode\hbox to3em{\hrulefill}\thinspace}
\providecommand{\MR}{\relax\ifhmode\unskip\space\fi MR }
\providecommand{\MRhref}[2]{%
  \href{http://www.ams.org/mathscinet-getitem?mr=#1}{#2}
}
\providecommand{\href}[2]{#2}

\end{document}